\documentclass[leqno,final]{siamltex}
\usepackage{amsmath}
\usepackage{graphicx}
\usepackage{amssymb}
\usepackage{color}

\newtheorem{remark}{Remark}

\newcommand{\cP}{\mathcal{P}}
\newcommand{\cQ}{\mathcal{Q}}

\newcommand{\cT}{\mathcal{T}}
\renewcommand{\div}{\mbox{\rm div\,}}
\newcommand{\curl}{\mbox{\rm curl\,}}

\newcommand{\cK}{\mathcal{K}}

\newcommand{\cF}{\mathcal{F}}
\newcommand{\cV}{\mathcal{V}}
\newcommand{\cW}{\mathcal{W}}

\newcommand{\mP}{\mathbb{P}}
\newcommand{\mV}{\mathbb{V}}
\newcommand{\mW}{\mathbb{W}}
\newcommand{\R}{\mathbf{R}}

\newcommand{\veps}{\varepsilon}
\newcommand{\Ome}{\Omega}
\newcommand{\p}{\partial}
\newcommand{\nab}{\nabla}
\def\esssupT{\underset{t\in [0,T]}{\mbox{\rm ess sup }}}

\begin{document}

\title{Analysis of Fully Discrete Mixed Finite Element Methods for Time-dependent Stochastic Stokes Equations
	with Multiplicative Noise}
\markboth{XIAOBING FENG AND HAILONG QIU}{MIXED FINITE ELEMENT METHODS FOR STOCHASTIC STOKES EQUATIONS}

\author{
Xiaobing Feng\thanks{Department of Mathematics, The University of Tennessee,
Knoxville, TN 37996, U.S.A. ({\tt xfeng@math.utk.edu})
The work of the this author was partially supported by
the NSF grants DMS-1318486 and DMS-1620168.}
\and
Hailong Qiu\thanks{School of Mathematics and Physics, Yancheng Institute of Technology,
Yancheng, 224051, China. ({\tt qhllf@163.com})
The work of the this author was partially supported by the NSF of China grant 11701498.}
}

\maketitle

\begin{abstract}
This paper is concerned with fully discrete mixed finite element approximations of
the time-dependent stochastic Stokes equations with multiplicative noise. A prototypical method, which
comprises of the Euler-Maruyama scheme for time discretization and the Taylor-Hood mixed element
for spatial discretization is studied in detail. Strong convergence with rates is established not only for
the velocity approximation but also for the pressure approximation (in a time-averaged fashion).
A stochastic inf-sup condition is established and used in a nonstandard way to
obtain the error estimate for the pressure approximation in the time-averaged fashion.
%
Numerical results are also provided to validate the theoretical
results and to gauge the performance of the proposed fully discrete mixed finite methods.
\end{abstract}



\section{Introduction}\label{sec-1}
We consider the following time-dependent stochastic Stokes equations for viscous incompressible fluids:
\begin{subequations}\label{eq1.1}
\begin{alignat}{2} \label{eq1.1a}
du &=\bigl[\nu\Delta u-\nabla p +f\bigr] dt +B(\cdot,u) dW(\cdot)  &&\qquad\mbox{a.s. in}\, D_T:=(0,T)\times D,\\
\div u &=0 &&\qquad\mbox{a.s. in}\, D_T,\label{eq1.1b}\\
u &=0 &&\qquad\mbox{a.s. on}\, \partial D_T:=(0,T)\times \partial D, \label{eq1.1c}\\
u(0)&=u_0 &&\qquad\mbox{a.s. in}\, D,\label{eq1.1d}
\end{alignat}
\end{subequations}
where $u$ and $p$ denote respectively the velocity field and the pressure of the fluid, and
$f$ is a given source field. $D\subset \R^d (d=2,3)$ is a bounded domain with boundary $\p D$,
$B: [0,T]\times [H^1(D)]^d\to L_0(L^2(D), [L^2(D)]^d)$ and $\{W(t); t\geq 0\}$ is an $[L^2(D)]^d$-valued
$Q$-Wiener process (see Section \ref{sec-2} for their precise definitions).

The system of equations \eqref{eq1.1a}, which is called stochastic Stokes equations/system, is
a simplified version of the stochastic Navier-Stokes model for turbulent fluids (cf. \cite{BT1973,Bensoussan95}
and the references therein) by omitting the nonlinear term $-(u\cdot \nabla)u\, dt$ on the right-hand side
of \eqref{eq1.1a}. The stochastic term $B(\cdot,u)\, dW$, which often is called the noise, adds a
solution-dependent source term to the corresponding deterministic Stokes (and the Navier-Stokes)
system.  When $B\equiv 0$, \eqref{eq1.1a} reduces to the time-dependent deterministic
Stokes equations \cite{Temam01}. A motivation for adding such a noise term is to allow the
stochastic models to capture the turbulence
phenomenon with choosing a "right" operator $B$ and a Wiener process $W$ by numerical simulation.
Since the Stokes system \eqref{eq1.1a}--\eqref{eq1.1d} is a simplification
of the more complicated stochastic Navier-Stokes system,
all the results about mathematical theory established for the
corresponding Navier-Stokes system clearly apply to the Stokes system.
We refer the interested reader to \cite{BT1973, Bensoussan95, Flandoli_Gatarek95, HM06, CCG10} and the
references therein for detailed discussions about solution concepts, well-posedness and regularity of
solutions of the stochastic Navier-Stokes problems with various types of noises.
We note that system \eqref{eq1.1} is nonlinear because $B(u)$ is a nonlinear function in $u$.

Since there exists a large amount of literature on numerical methods
for the deterministic Stokes and Navier-Stokes equations, it is natural to try to adapt those
successful numerical methods for solving problem \eqref{eq1.1}.
Indeed, with some special care on discretizing the stochastic term $B(\cdot,u)\, dW$,
all other terms in \eqref{eq1.1a} can be discretized in the same way as done in the deterministic case.
However, since in the stochastic case the solution $u$ is a Hilbert
space-valued stochastic process and spatial norms of the velocity are only H\"older continuous in $t$,
all the deterministic analysis techniques and machineries which require the differentiability of the solution  in $t$ are not applicable in the stochastic case.
Moreover, for saddle point problem \eqref{eq1.1}, it is known \cite{LRS03} that the pressure
process $\{p(t); 0\leq t\leq T\}$ has very weak regularities (it is only a measure), which makes
the numerical analysis of \eqref{eq1.1} quite challenging, especially for the mixed finite
element discretization where the velocity and pressure are approximated simultaneously.
As a result, the pool of rigorous numerical methods (i.e.,  those with support of convergence analysis)
for the stochastic Stokes and Navier-Stokes equations is limited and those results only become available recently.

In \cite{CHP12} Carelli {\em et al.} proposed a Chorin-type time-splitting method
for problem \eqref{eq1.1} to address the subtle interplay of noise and pressure. Strong convergence
with rates was proved for the velocity approximation with sinusoidal noises.
Although pressure approximation was also constructed and computed by the Chorin-type method, its convergence
was not addressed in \cite{CHP12}. In addition, the semi-discrete in time Euler-Maruyama scheme
was also considered and used as a tool in the analysis of the Chorin-type method, but its convergence
analysis was not addressed.
In \cite{CP12} Carelli and Prohl proposed an implicit and a semi-implicit
time discretizations for the stochastic Navier-Stokes problem in $2$-D with sinusoidal noises.
Strong convergence with rates was also established for the velocity approximation.
A fully discrete mixed finite element scheme was also considered and strong convergence
with rates was also proved for the velocity approximation. As noted
in \cite{CP12}, the interaction of Lagrange multipliers with the stochastic forcing in the scheme
limits the accuracy of general discretely LBB-stable space discretizations. Strategies to overcome
this difficulty were also proposed in \cite{CP12} although the convergence of the pressure approximation
was not addressed. Other recent works for stochastic Navier-Stokes equations include
an iterative splitting scheme which was proposed in \cite{BBM14}, a strong convergence in probability
was established in the 2-D case for the velocity approximation. In a very recent paper \cite{BM18}, the authors
proposed another time-splitting scheme and proved its strong $L^2$ convergence for the velocity approximation.
In \cite{YDG10} a posterior error estimates were studied for a fully discrete divergence-free
finite element method for the 2-D stochastic Navier-Stokes equations, both upper and lower a
posterior error bounds were established for the velocity approximation in the paper. 
%
%
The paper \cite{BCP12} by Brze\'{z}niak, {\em et al.}, which perhaps is closest to this paper,
proposed two time-stepping schemes for mixed
finite element spatial discretizations of the stochastic Navier-Stokes equations with general
multiplicative noise and established the convergence for the velocity approximation
(as a function sequence) to weak martingale solutions in 3-D and to strong solutions in 2-D
using the compactness argument. Since the analysis was done in the space of discretely
solenoidal functions which allows to eliminate the discrete pressures from the schemes,
consequently, the convergence of the pressure approximations was bypassed in \cite{BCP12};
in addition, no rate of convergence was presented for the velocity approximation.
To the best of our knowledge, no convergence (hence, no rate of convergence) for a pressure approximation 
has been reported in the literature for system \eqref{eq1.1}.

The primary goal of this paper is to analyze convergence behaviors of the mixed
finite element method for the stochastic Stokes problem \eqref{eq1.1} with general multiplicative noise. Specifically, we shall establish strong convergence with rates not only for the velocity approximation
but also for the pressure approximation (in a time-averaged fashion). A secondary goal of the paper
is to use \eqref{eq1.1} as a prototypical example to develop numerical analysis techniques
which can be useful for analyzing mixed finite element approximations of the stochastic Navier-Stokes
equations and possibly other nonlinear stochastic PDEs.

The remainder of this paper is organized as follows. In Section \ref{sec-2} we introduce
function and space notation and some background materials for problem \eqref{eq1.1}
and establish a few preliminary results such as the stochastic
inf-sup condition and H\"older continuity in time of the solution in various spatial norms.
These results plays an important role in the error analysis of this paper.
In Section \ref{sec-3} we state the Euler-Maruyama time-stepping scheme for problem \eqref{eq1.1} and
derive some error estimates for both the velocity and pressure approximations.
In Section \ref{sec-4} we formulate a fully discrete mixed finite element method
which uses the prototypical Taylor-Hood mixed finite element method for spatial discretization.
The highlight of this Section is to establish the following error estimates for the numerical
solution $(u_h^n, p_h^n)$:
\begin{align}\label{eq1.4}
\max_{1\leq n\leq N} \Bigl(E\Bigl[\|u(t_n)-u^n_h\|^2_{L^2}\Bigl]\Bigl)^{\frac12}
&+\Bigl(E\Bigl[k\sum_{n=1}^{N}\|\nabla (u(t_n)-u^n_h)\|^2_{L^2}\Bigr]\Bigr)^{\frac12}
\\\nonumber  &\leq \hat{C}_1k^{\frac12}+ \hat{C}_2 k^{-\frac12} h,\\\label{eq1.5}
\max_{1\leq n\leq N} \Bigl(E\Bigl[\|\nabla (u(t_n)-u^n_h)\|^2_{L^2}\Bigr]\Bigr)^{\frac12}
&+\Bigl( E\Bigl[ k\sum_{n=1}^{N}\|A (u(t_n)-u^n_h)\|^2_{L^2}\Bigr] \Bigr)^{\frac12}\\ \nonumber
&\leq \hat{C}_3 k^{\frac12}+ \hat{C}_4 k^{-\frac12} h, \\
\label{eq1.5a}
\max_{1\leq n\leq N} \Bigl(E\Bigl[\|A(u(t_n)-u^n_h)\|^2_{L^2}\Bigr]\Bigr)^{\frac12}
&+\Bigl( E\Bigl[ k\sum_{n=1}^{N}\|A^{\frac{3}{2}}(u(t_n)-u^n_h)\|^2_{L^2}\Bigr] \Bigr)^{\frac12}\\ \nonumber
&\leq \hat{C}_5 k^{\frac12}+ \hat{C}_6 k^{-\frac12} h, \\\label{eq1.5b}
\max_{1\leq m\leq N} E\Bigl[\Bigl\|\int_0^{t_m} p(s)ds-k\sum^m_{n=1} p^n_h \Bigr\|_{L^2}\Bigr]
&\leq \hat{C}_7 k^{\frac12}+\hat{C}_8 k^{-\frac12} h.
\end{align}
%
It should be noted that the (bad) $k^{-\frac12}$ factor in the above error estimates
is due to the low regularity of the pressure $p$ and the simultaneous approximation property of
the mixed method. Obviously, compared to the error estimates for mixed finite element approximations of
the deterministic Stokes problem
(cf. \cite{Girault_Raviart86}), the above estimates seem inferior, however, our numerical experiments
indicate that these estimates are in fact sharp for the stochastic Stokes problem \eqref{eq1.1}.
The above theoretical results on one hand reveal the insight of the (damaging)
effect of the noise on the performance of the standard mixed finite element method for the stochastic
Stokes problem and on the other hand suggest that modifications and improvements must be done
to the standard mixed finite element method in order to improve its performance so that the modified method
(such as mixed-primal finite element method, stabilized finite element method)
could produce competitive velocity approximations with that of divergence-free finite element methods (cf. \cite{CP12}).
Finally, in Section \ref{sec-5}, we provide some numerical experiments to validate our
theoretical results and to gauge the performance of the proposed numerical method.

\section{Preliminaries}\label{sec-2}
\subsection{Notation and assumptions}\label{sec-2.1}
Standard function and space notation will be adopted in this paper. In particular, for a given
positive integer $m$, let $H^m(D)$ denote the standard Sobolev space consisting of all
real-valued functions whose up to $m$th order weak derivatives are $L^2$-integrable
on $D\subset \R^d$, and $\|\cdot\|_{H^m}$ denotes its norm.
Let $H^1_0(D)$ be the subspace of $H^1(D)$ whose functions have zero trace on $\p D$,
$H^0(D):=L^2(D)$ and $(\cdot,\cdot):=(\cdot,\cdot)_D$ denote the standard $L^2$-inner product.
Let $(\Omega,\cF,\cF_t,\mP)$ be a probability space with $\sigma$-algebra $\cF$, the normal filtration
$\cF_t$ and the probability measure $\mP$. For a random variable $v$ defined on $(\Omega,\cF,\cF_t,\mP)$,
let $E[v]$ denote the expected value of $v$. We also let $(\Omega,\cF, \{\cF_t\},\mP)$ be a complete
probability space with continuous filtration $\{\cF_t\subset \cF; t\geq 0\}$.
a.s. means {\em almost surely} with respect to the probability measure $\mP$.
For a Banach space $Y$, let $L^p(0, T; Y)$ denote the time-space function space
  endowed with following norm:
 \[
\|w\|_{L^p(0,T;Y)} :=\left\{\begin{array}{ll}
 \bigl(\int_0^T\|w\|_{Y}^pdt\bigr)^{\frac1p}   &\,\mbox{if}\, \, \,\, 1\leq p< \infty,\\ \\
\esssupT\|w\|_{Y}  &\,\mbox{if}\,  \, \, \, p= \infty.
 \end{array}\right.
\]
 We shall often use the abbreviated notation $L^p(Y):=L^p(0, T; Y)$ for convenience.

For a normed vector space $X$ with norm $\|\cdot\|_{X}$, let $[X]^d$ denote the space of
all $d$-vector-valued mappings whose components belong to $X$. Define the Bochner space
\[
L^p(\Omega,X):=\bigl\{v:\Omega\rightarrow X;\, E[\|v\|_{X}^p] <\infty \bigr\}
\]
and the norm
\[
\|v\|_{L^p(\Omega,X)}:=\Bigl(E\Bigl[\|v\|_X^p\Bigr]\Bigr)^{\frac1p}, \qquad 1<p<\infty.
\]
We also introduce the following special space notation:
\begin{align*}
&\cV:=[H^1_0(D)]^d, \cV_{*}:=[H^2_0(D)]^d, \quad \cW=L^2_0(D):=\{v\in L^2(D); \, (v,1)_D=0 \}, \\
&\cV_0:=\bigl\{v\in \cV;\,\div v=0 \mbox{ in }D \bigr\}, \quad \mV:=L^2(\Omega,\cV),
\quad \mW:=L^2(\Omega, \cW).
\end{align*}

To give a meaning to the stochastic term $B(\cdot, u)\, dW$, we need to recall the
definition of Hilbert space-valued $Q$-Wiener process $W$.  Let $Q$ be a non-negative and symmetric
linear operator from $L^2(D)$ to itself. Assume that $Q$ has a set of eigenvalues and
eigenfunctions $\{ (\lambda_j, q_j) \}_{j\geq 1}$ such that $\{q_j\}_{j\geq 1}$ forms an orthonormal basis
for $L^2(D)$. Let $\{\beta_j(t); t\geq 0\}_{j\geq 1}$ be a sequence of independent identically
distributed (iid) real-valued Brownian motions (or Wiener processes) adapted to
$\{\cF_t\}$. Then an $L^2(D)$-valued $Q$-Wiener process
$W=\{W(t); t \geq 0\}$ on $(\Omega,\cF, \{\cF_t\},\mP)$ is defined as
\begin{equation}\label{eq2.1}
W(\cdot,t)=\sum_{j=1}^{\infty}\sqrt{\lambda_j} q_j(\cdot) \beta_j(t) \qquad\mbox{a.s. in } D.
\end{equation}

Let $\cK_{*}:= Q^{\frac{1}{2}}(L^2(D))$ be the Cameron-Martin space for the Wiener measure on $C([0, T ]; L^2(D))$
 equipped with the scalar product $(\varphi, \psi)_{\cK_{*}} = (Q^{-\frac{1}{2}}\varphi, Q^{-\frac{1}{2}}\psi)_{L^2(D)}$
 for all $\varphi, \psi\in L^2(D)$, where $Q^{-\frac{1}{2}}$ denotes the pseudoinverse of  $Q^{\frac{1}{2}}$.
Let $\cK:=L_0(L^2(D);[L^2(D)]^d)$ be the Banach space of linear operators
from $L^2(D)$ to $[L^2(D)]^d$ which have finite Hilbert-Schmidt norms denoted by $\|\cdot\|_{\cK}$. For any $1<p<\infty$,
let $M^p_{\cF_t}(\Omega, L^p(0,T; \cK))$ be the subspace of the Bochner space
$L^p(\Omega, L^p(0,T; \cK))$ whose mappings are $\{\cF_t\}$-adapted.
Then for any $\varphi\in M^2_{\cF_t}(\Omega, L^2(0,T; \cK))$, the stochastic integral
$\int_0^t\varphi(s)\,dW(s)$ for $0\leq t\leq T$ is defined as an $[L^2(D)]^d$-valued function by
\begin{equation}\label{ito_int}
\Bigl(\int_0^t\varphi(s)\,dW(s),\chi\Bigr)
= \lim_{J\rightarrow\infty}\sum^J_{j=1}  \sqrt{\lambda_j} \int_0^t\bigl(\varphi(s)q_j,\chi\bigr)\, d\beta_j(s)
\quad \forall \chi\in [L^2(D)]^d, \, a.s.
\end{equation}
Note that the stochastic integral on the right-hand side is understood in the It\^o's sense.
It is well-known \cite{PZ92} that $\int_0^t\varphi(s)dW(s)$ is an $\{\cF_t\}$-martingale and there holds
the following It\^o's isometry:
\begin{equation}\label{eq2.2}
E\Bigl[\Bigl\|\int_0^T\varphi(s)\,dW(s)\Bigr\|_{L^2(D)}^2\Bigr]
=E\Bigl[\int_0^T\|\varphi(s)\|^2_{\cK}\,ds\Bigr].
\end{equation}

The above definition of stochastic integrals suggests that $B(\cdot,u)$ needs to belong to
$M^2_{\cF_t}(\Omega, L^2(0,T;\cK))$ in order to give a meaning to the stochastic term in \eqref{eq1.1a}.
Indeed, in this paper we assume that $B:[0,T]\times [L^2(D)]^d\rightarrow L^2(\Omega,\cK)$ is
H\"older-Lipschitz continuous in time and has a linear growth in the second argument in the sense that
there exists a constant $C_T > 0$ such that $\mP$-a.s.
\begin{subequations}
\begin{align}\label{eq2.6}
\|D^{j}(B(s,v)-B(t,w))\|_{\cK} &\leq C_T\bigl[|s-t|^{\frac12} +\|D^{j}(v-w)\|_{L^2}\bigr], \quad  j=0,1,2,\\    
  \label{eq2.6a}
 |D^{i}B(t,v)| &\leq  C_T \bigl(1+\|D^{i}v\|_{L^2} \bigr),\ \ \  i=0,1,2,3,4 
\end{align}
\end{subequations}
for any $v, w \in [L^2(D)]^d, \, s,t\in [0,T]$, for $v\in  [L^2(D)]^d$.
\begin{equation}\label{eq2.7}
B(\cdot,v)\in  L^2(\Omega; L^2(0,T; \cK)).
\end{equation}

Finally, we assume that $D\subset \R^d (d=2,3)$ is a bounded domain such that there is a
unique strong solution $(v,\chi)\in \cV_0\cap[H^2(D)]^d\times L^2_0(D)\cap H^1(D)$ to the following deterministic
stationary Stokes problem:
\begin{alignat*}{2}
-\nu\Delta v+\nabla \chi &=g  &&\qquad\mbox{in }D ,\\
\div v&=0 &&\qquad \mbox{in }D,\\
v &=0  &&\qquad \mbox{on } \partial D,
\end{alignat*}
which satisfies the estimate
\begin{equation*}
\|v\|_{H^2(D)}+\|\chi\|_{H^1(D)}\leq C\|g\|_{L^2(D)}
\end{equation*}
for any $g\in L^2(D)$. It is well-known \cite{Temam01} that the above regularity holds
if $D$ is a convex polygonal domain in 2-D,  while in $3$-D it holds for $C^2$-domains.

\subsection{Properties of variational weak solutions}\label{sec-2.2}
In this subsection we first recall the variational weak solution definition for problem \eqref{eq1.1}.
We then prove a stochastic inf-sup condition for the  mixed term and establish
some H\"older continuity (in time) of various spatial norms of weak solutions. These
auxiliary results will play an important role in our convergence analysis to be given
in the subsequent sections.

\begin{definition}\label{def2.1} (\cite{LRS03})
Suppose that $u_0\in L^2(\Omega, \cV_0)$ and $f\in L^2(\Omega,L^{\infty}(0,T;$ $[L^2(D)]^d)$.
A pair of $\{\cF_t\}$-adapted stochastic processes $\{(u(t), p(t)); 0\leq t\leq T\}$ is called
a weak solution of \eqref{eq1.1} if $(u,p)\in  L^2\bigl(\Omega, L^\infty(0,T; [L^2(D)]^d) \cap L^2(0,T; \cV) \cap C(0,T; \cV^*) \bigr)
\times   L^2\bigl(\Omega, L^2(0,T; W^{-1,\infty}(D))\bigr)$ and satisfy $\mP$-a.s
\begin{subequations}\label{eq2.8}
\begin{align}\label{eq2.8a}
&\bigl(u(t),\psi \bigr) +\int_0^t \nu a(u(s),\psi)\,  ds +  b\Bigl(\psi, \int_0^t p(s)\, ds \Bigr)   \\ \nonumber
&\qquad =(u_0,\psi)+\int_0^t\bigl(f(s),\psi\bigr)\, ds+\Bigl(\int_0^t B(s,u(s))\, dW(s),\psi\Bigr) &&\qquad\forall \psi\in \cV,\\
& b(u,q) =0 \qquad\forall q\in \cW  &&  \label{eq2.8b}
\end{align}
\end{subequations}
for all $t\in (0,T]$. Where the bilinear forms $a(\cdot,\cdot)$ and $b(\cdot,\cdot)$ are defined as follows:
\begin{subequations} \label{eq2.12}
\begin{alignat}{2}\label{eq2.12a}
a(v,w) &:= \bigl(\nabla v, \nabla w\bigr) &&\qquad \forall v,w\in \cV,\\
b(v,q) &:= \bigl(\div v, q\bigr)  &&\qquad\forall v\in \cV,\,  q\in \cW.
\end{alignat}
\end{subequations}

\end{definition}

\begin{remark}\label{rem2.1}
	We write $b\Bigl(\psi, \int_0^t p(s)\, ds \Bigr)$ instead of  $\int_0^t b(\psi, p(s))\, ds$ in  \eqref{eq2.8a} because the regularity of $p$ is too weak to allow the latter to make sense unless more
	smoother test function $\psi$ is used. On the other hand, it could be shown that $\int_0^t p(s)\, ds \in L^2\bigl(\Omega, L^2(D\times (0,T))\bigr)$ so the former is well defined.  Thus, in the remainder of
	this paper, the pressure process $\{p(t)\}_{t\geq 0}$ will always appear in the time integral form.
\end{remark}

It should also be noted that other solution notions such as mild and strong solutions
have also be introduced and studied in the literature for problem \eqref{eq1.1}
(cf. \cite{PZ92,Chow07} and the references therein).
In addition, nonhomogeneous boundary conditions can be dealt with through standard lifting arguments.
In this paper, we only consider homogeneous boundary conditions to avoid the technicalities.

Next, we state a stochastic inf-sup condition for the bilinear form $b(\cdot, \cdot)$.

\begin{lemma}\label{lem2.1}
There exists a positive constant $\beta$ such that
\begin{align}\label{eq2.15}
\sup_{v\in \mV}\frac{E[b(v,q)]}{\|v\|_{\mV}}\geq \beta\|q\|_{\mW} \qquad\forall q\in \mW.
\end{align}
\end{lemma}

\begin{proof}
Below we give a proof for the sake of completeness.
The proof follows the same lines as that for the deterministic inf-sup condition
(cf. \cite{ASV88,Girault_Raviart86}), we only sketch the main ideas and steps of the proof
in the case when $D$ is a 2-D bounded domain with smooth or convex polygonal boundary $\p D$.

For any fixed  $q\in \mW$, the first step  of the proof is to construct a random
field $v_1\in L^2(\Omega, [H^1(D)]^d)$
such that $\div v_1=q$ in $L^2(\Omega\times D)$ and $v_1\cdot n=0$ a.s. on $\p D$, where
$n$ denotes the unit outward normal to $\p D$. In addition, there exists a positive constant
$c_1$ such that
\begin{align}\label{eq2.16}
\|v_1\|_{L^2(\Ome,[H^1(D)]^d )}  \leq c_1 \|q\|_\mW.
\end{align}
The desired random field $v_1$ can be chosen as $v_1=-\nabla \phi$, where $\phi \in L^2(\Omega,H^2(D))$
is the solution of the following random Poisson problem:
\begin{align*}
-\Delta \phi &=q \qquad\mbox{a.s. in } D,\\
\frac{\p \phi}{\p n} &=0 \qquad\mbox{a.s. on } \p D,\\
(\phi, 1) &=0  \qquad\mbox{a.s.}
\end{align*}
By the elliptic PDE theory we know that there exists a unique solution $\phi$, moreover, there exists
a positive constant $c_1$ such that
\begin{align*}
\|\phi\|_{H^2(D)}\leq c_1 \|q\|_{L^2(D)}\qquad \mbox{a.s.}
\end{align*}
which clearly implies \eqref{eq2.16} with $v_1=-\nabla \phi$.

The second step of the proof is to construct a random field $v_2\in L^2(\Omega, [H^1(D)]^d)$
such that $\div v_2=0$ in $L^2(\Omega\times D)$,
$v_2\cdot \tau= -v_1\cdot \tau$ and $v_2\cdot n=0$ a.s. on $\p D$, where $\tau$ denotes
the positively oriented unit tangent vector to $\p D$. In addition, there exists a positive
constant $c_2$ such that
\begin{align}\label{eq2.17}
\|v_2\|_{L^2(\Ome,[H^1(D)]^d )}  \leq c_2 \|v_1\|_{L^2(\Ome,[H^1(D)]^d )}.
\end{align}
One such a random field is $v_2= \curl \psi$, where $\psi\in L^2(\Omega,H^2(D))$ satisfies
\begin{alignat*}{2}
\Delta^2\psi & =0  &&\qquad\mbox{a.s. in }  D,\\
\frac{\p \psi}{\p n} &=v_1\cdot \tau  &&\qquad\mbox{a.s. on } \p D,\\
\psi&=0  &&\qquad\mbox{a.s. on } \p D,
\end{alignat*}
and
\[
\|\psi\|_{H^2(D)} \leq c_2 \|v_1\|_{H^1(D)} \qquad\mbox{a.s.}
\]
for some positive constant $c_2$. Note that the existence of $\psi$ is guaranteed by the elliptic PDE theorey.
The above estimate immediately infers \eqref{eq2.17}.

The third step of the proof is to define $z=v_1+v_2$. It is easy to check that
$z\in L^2(\Ome,[H^1(D)]^d )$,
$\div z=q$ in $L^2(\Omega\times D)$ and $z=0$ a.s. on $\p D$, hence, $z\in \mV$.  Moreover,
it follows from \eqref{eq2.16} and \eqref{eq2.17} and the triangle inequality that
\begin{align}\label{eq2.18}
\|z\|_{\mV} 
\leq (1+c_2) \|v_1\|_{L^2(\Ome,[H^1(D)]^d )} \leq (1+c_2)c_1\|q\|_{\mW}.
\end{align}

Finally,  it follows from \eqref{eq2.18} that
\begin{align*}
\sup_{v\in \mV}\frac{E[b(v,q)]}{\|v\|_\mV}
\geq\frac{E[b(z,q)]}{\|z\|_\mV} =\frac{E\bigl[\|q\|_{L^2(D)}^2 \bigr]}{\|z\|_\mV}
\geq \frac{1}{(1+c_2)c_1}\|q\|_\mW. 
\end{align*}
Hence, \eqref{eq2.15} holds with $\beta=\frac{1}{(1+c_2)c_1}$. The proof is complete.
\end{proof}

The next theorem establishes  some H\"older continuity (in time) of the velocity field $u$
in various spatial norms.

\begin{theorem}\label{thm2.2}
Let $(u,p)$ be the weak solution to problem \eqref{eq1.1}, and assume that
$u\in L^2(\Omega, L^\infty(0,T; \cV\cap [H^3(D)]^d)),
\nabla u_j, A u, A^{\frac{3}{2}} u \in L^2(\Omega; C([0,T]; [L^2(D)]^d))$ for $1\leq j\leq d$, 
and $f\in L^2(\Omega;L^{\infty}(0,T,H^1(D)^d))$.
 Then there hold
\begin{subequations}
\begin{align}\label{eq2.20a}
& E\bigl[\|u(t)-u(s)\|^2_{L^2}\bigr]+\nu E\Bigl[\int_s^t \|\nabla(u(\xi)-u(s))\|_{L^2}^2\Bigr]d\xi
\leq C_1|t-s|,\\ \label{eq2.20b}
& E\bigl[\|\nabla (u(t)-u(s))\|^2_{L^2}\bigr]+\nu E\Bigl[\int_s^t \|A(u(\xi)-u(s))\|_{L^2}^2\Bigr]d\xi
\leq C_2|t-s|, \\ \label{eq2.20c}
&  E\bigl[\|A (u(t)-u(s))\|^2_{L^2}\bigr]+\nu E\Bigl[\int_s^t \|A^{\frac{3}{2}}(u(\xi)-u(s))\|_{L^2}^2\Bigr]d\xi
\leq C_3|t-s|
\end{align}
\end{subequations}
for any $0\leq s,t\leq T$,  where
\begin{align*}
&C_1 =\Bigl( ( 2\nu+2C_0C_T) E[\|\nabla u\|_{L^\infty(L^2)}^2]
+\frac{2}{\nu}E[\|f\|^2_{L^{\infty}(H^{-1})}]\Bigr)(1+2C_TT)e^{2C_TT},\\
&C_2 =\Bigl((2\nu+2C_0C_T)E\bigl[\|Au\|_{L^\infty(L^2)}^2\bigr]
+\frac{2}{\nu}E\bigl[\|f\|^2_{L^\infty(L^2)}\bigr] \Bigr) (1+2C_TT)e^{2C_TT},\\
&C_3 =\Bigl((2\nu+2C_0C_T)E\bigl[\|A^{\frac{3}{2}}u\|_{L^\infty(L^2)}^2\bigr]
+\frac{2}{\nu}E\bigl[\|f\|^2_{L^\infty(H^1)}\bigr] \Bigr)(1+2C_TT)e^{2C_TT},
\end{align*}
where $A:\cV\cap [H^2(D)]^d\to \cV_0$ denotes the Stokes operator (cf. \cite{Temam01}) and
$C_0$ is the constant in the Poincar\'e inequality for $H^1(D)$ functions.
\end{theorem}


\begin{proof}
We first note that estimate \eqref{eq2.20a} was obtained in \cite{BCP12, CP12} using
	the semigroup approach, the results of \eqref{eq2.20b}--\eqref{eq2.20c} seem new. Below we
	give a proof for each estimate using the It\^{o}'s formula approach.

{\em Step 1:} For a fix $s\in(0,T]$, by the definition of weak solutions we obtain
\begin{align*}
\bigl(u(t)-u(s),v\bigr) +\nu\int_s^t a(u,v)\,d\xi
=\int_s^t(f,v)\,d\xi+\Bigl(\int_s^t B(\xi,u)\,dW(\xi),v\Bigr)
\end{align*}
for any $v\in \cV_0$. Now apply It\^o's formula to $\Phi(u(t)):=\|u(t)-u(s)\|^2_{L^2}$ to get
\begin{align}\label{eq2.22}
&\|u(t)-u(s)\|^2_{L^2}+2\nu\int_s^t \bigl(\nabla u(\xi),\nabla(u(\xi)-u(s)) \bigr)\,d\xi \\
&\qquad
=2\int_s^t \bigl(f(\xi),u(\xi)-u(s)\bigr)\,d\xi
+2 \int_s^t\Bigl(B(\xi,u(\xi))\,dW(\xi),u(\xi)-u(s)\Bigr) \nonumber\\
&\qquad\qquad
+\int_s^t\|B(\xi,u(\xi))\|^2_{\cK}\,d\xi. \nonumber
\end{align}
Then we have
\begin{align}\label{eq2.23}
&\|u(t)-u(s)\|^2_{L^2}+2\nu\int_s^t \|\nabla u(\xi)-u(s)\bigr|_{L^2}^2\,d\xi\\\nonumber
&\qquad
=2\nu\int_s^t \bigl(\nabla u(s),\nabla(u(\xi)-u(s))\bigr)\,d\xi
+2\int_s^t \bigl(f(\xi),u(\xi)-u(s)\bigr)\,d\xi\\ \nonumber
&\qquad\qquad
+\int_s^t \|B(\xi,u(\xi))\|^2_{\cK}\,d\xi+2 \int_s^t\Bigl(B(\xi,u(\xi))\,dW(\xi),u(\xi)-u(s)\Bigr) \\ \nonumber
&\qquad
=:I_1+I_2+I_3+I_4.
\end{align}
We then bound each $I_i$ for $i=1,2,3,4$ separately below.

To bound $I_1$, by Schwarz and Young's inequality we get
\begin{align}\label{eq2.24}
2\nu\int_s^t \bigl(\nabla u(s), &\nabla(u(\xi)-u(s))\bigr)\,d\xi
\leq 2\nu\int_s^t \|\nabla u(s)\|_{L^2} \|\nabla(u(\xi)-u(s)))\|_{L^2}\,d\xi \\ \nonumber
&\leq 2\nu\Bigl(\int_s^t \|\nabla (u(\xi)-u(s))\|_{L^2}^2\,d\xi\Bigr)^{\frac12}
\Bigl(\int_s^t \|\nabla u(s)\|_{L^2}^2\,d\xi\Bigr)^{\frac12}\\ \nonumber
&\leq \frac{\nu}{2}\int_s^t \|\nabla(u(\xi)-u(s))\|^2_{L^2}\,d\xi +2\nu\|\nabla u(s)\|_{L^2}^2\, |t-s|.
\end{align}

Similarly, $I_2$ can be bounded as follows:
\begin{align}\label{eq2.25}
2\int_s^t \bigl(f(\xi), &u(\xi)-u(s)\bigr)\, d\xi
\leq 2\Bigl(\varepsilon\int_s^t \|\nab(u(\xi)-u(s))\|_{L^2}^2\,d\xi \\ \nonumber
&\hskip 1.2in
+\frac{1}{4\varepsilon}\int_s^t \|f(\xi)\|_{H^{-1}}^2\,d\xi\Bigr)\\ \nonumber
&\leq 2\Bigl(\varepsilon \int_s^t \|\nabla (u(\xi)-u(s))\|_{L^2}^2\, d\xi
+\frac{1}{4\varepsilon}\int_s^t \|f(\xi)\|_{H^{-1}}^2\, d\xi\Bigr)\\ \nonumber
&\leq\frac{\nu}{2}\int_s^t \|\nabla(u(\xi)-u(s))\|_{L^2}^2\, d\xi +\frac{2}{\nu} \|f\|^2_{L^{\infty}(H^{-1})}\,|t-s|,
\end{align}
where we set $\varepsilon=\frac{\nu}{4}$.

To bound $I_3$, we have
\begin{align}\label{eq2.26}
&\int_s^t \|B(\xi,u(\xi))\|^2_{\cK}\,d\xi
=\int_s^t \|B(\xi,u(\xi))-B(s,u(s))+B(s,u(s))\|^2_{\cK}\,d\xi\\  \nonumber
&\qquad \leq 2\int_s^t\|B(\xi,u(\xi))-B(s,u(s))\|^2_{\cK}\, d\xi+2\int_s^t\|B(s,u(s))\|^2_{\cK}\, d\xi\\ \nonumber
&\qquad \leq 2C_T\int_s^t\|u(\xi)-u(s)\|^2_{L^2}\,d\xi+2C_T\int_s^t\|u(s)\|^2_{L^2}\, d\xi\\ \nonumber
&\qquad \leq 2C_T\int_s^t\|u(\xi)-u(s)\|^2_{L^2}\,d\xi+2C_T\|u(s)\|^2_{L^2}\,|t-s|.
\end{align}

Finally, on noticing that $I_4$ is a martingale, then we have $E[I_4]=0$. Now combining
\eqref{eq2.24}--\eqref{eq2.26} and taking the expectation we obtain
\begin{align*} 
&E\bigl[\|u(t)-u(s)\|^2_{L^2}\bigr] +\nu E\Bigl[\int_s^t \|\nabla(u(\xi)-u(s))\|_{L^2}^2\,d\xi \Bigr] \\ \nonumber
&\qquad
\leq \Bigl( 2\nu E\bigl[\|\nabla u(s)\|_{L^2}^2\bigr] + 2C_T E\bigl[\|u(s)\|_{L^2}^2\bigr]
+\frac{2}{\nu} E\bigl[\|f\|^2_{L^{\infty}(H^{-1})}\bigr]\Bigr)\,|t-s|\\ \nonumber
&\qquad\qquad
+2C_TE\Bigl[\int_s^t\|u(\xi)-u(s)\|^2_{L^2}\,d\xi\Bigr].
\end{align*}
An application of Gronwall and Poincar\'e inequality yields
\begin{align*} 
&E\bigl[\|u(t)-u(s)\|^2_{L^2}\bigr] +\nu E\Bigl[\int_s^t \|\nabla(u(\xi)-u(s))\|_{L^2}^2\,d\xi \Bigr] \\ \nonumber
&\,\, \leq\Bigl( ( 2\nu+2C_0C_T) E[\|\nabla u\|_{L^\infty(L^2)}^2]
+\frac{2}{\nu}E[\|f\|^2_{L^{\infty}(H^{-1})}]\Bigr)(1+2C_TT)e^{2C_TT}\,|t-s|.
\end{align*}
Hence, \eqref{eq2.20a} holds.

\medskip
{\em Step 2:} To show the second inequality \eqref{eq2.20b},
we apply It\^o's formula to $\Psi(u(t)):=\|\nabla (u(t)-u(s))\|^2_{L^2}$ to get
\begin{align}\label{eq2.30}
&\|\nabla (u(t)-u(s))\|^2_{L^2}+2\nu\int_s^t \bigl(A u(\xi)-u(s),A(u(\xi)-u(s))\bigr)\,d\xi\\\nonumber
&\quad
=2\nu\int_s^t \bigl(A u(s),A(u(\xi)-u(s))\bigr)\,d\xi
+2\int_s^t \bigl(f(\xi),A(u(\xi)-u(s))\bigr)\,d\xi \\ \nonumber
&\qquad\quad
+\int_s^t\|\nabla B(\xi, u(\xi))\|^2_{\cK}\,d\xi
+2\int_s^t \bigl(\nabla B(\xi,u(\xi))\,dW(\xi),\nabla (u(\xi)-u(s))\bigr) \\
\nonumber
&\quad =:I_5+I_6+I_7+I_8.
\end{align}
We now bound each $I_i$ for $i=5,6,7,8$ as follows,

To bound $I_5$, we use Schwarz and Young's inequality to get
\begin{align}\label{eq2.31}
2\nu\int_s^t \bigl(Au(s),A(u(\xi) &-u(s))\bigr)\,d\xi
\leq 2\nu \int_s^t \|A (u(\xi)-u(s))\|_{L^2} \|A u(s)\|_{L^2}\,d\xi\\ \nonumber
&\leq \frac{\nu}{2}\int_s^t \|A(u(\xi)-u(s))\|^2_{L^2}\,d\xi +2\nu\|Au(s)\|_{L^2}^2\, |t-s|.
\end{align}
Similarly, we have
\begin{align}\label{eq2.32}
2\int_s^t \bigl(f(\xi),A (u(\xi) &-u(s))\bigr)\,d\xi
\leq 2\int_s^t \|f(\xi)\|_{L^2} \|A (u(\xi)-u(s))\|_{L^2}\,d\xi \\  \nonumber
&\leq \frac{\nu}{2}\int_s^t \|A(u(\xi)-u(s))\|_{L^2}^2\,d\xi
+\frac{2}{\nu}\|f\|_{L^{\infty}(L^2)}^2\,|t-s|.
\end{align}
\begin{align}\label{eq2.33}
\int_s^t \|\nabla B(\xi, &u(\xi))\|^2_{\cK}\,d\xi
=\int_s^t \|\nabla B(\xi,u(\xi)) \pm \nabla B(s,u(s))\|^2_{\cK}\,d\xi \\ \nonumber
&\leq 2C_T\int_s^t \|\nabla(u(\xi)-u(s))\|^2_{L^2}\,d\xi
+2C_T\|\nabla u(s)\|^2_{L^2}\,|t-s|.
\end{align}
Here we have used the Lipschitz continuity of $B(\cdot, u)$ with $j=1$.

Since $I_8$ it is a martingale, then $E[I_8]=0$. Combining \eqref{eq2.31}--\eqref{eq2.33} and
taking the expectation we get
\begin{align*} 
&E\bigl[\|\nabla (u(t)-u(s))\|^2_{L^2}\bigr]
+\nu E\Bigl[ \int_s^t \|A(u(\xi)-u(s))\|_{L^2}^2\,d\xi\Bigr]\\ \nonumber
&\qquad
\leq \Bigl( 2\nu E\bigl[\|Au(s)\|^2_{L^2}\bigr]+\frac{2}{\nu} E\bigl[\|f\|^2_{L^\infty(L^2)}\bigr]
+2C_TE\bigl[\|\nabla u(s)\|_{L^2}^2\bigr]\Bigr)\, |t-s|\\ \nonumber
&\qquad\quad
+2C_T E\Bigl[\int_s^t \|\nabla (u(t)-u(s))\|^2_{L^2}\, d\xi \Bigr].
\end{align*}
It follows from  Gronwall inequality  that
\begin{align*} 
&E\bigl[\|\nabla(u(t)-u(s))\|^2_{L^2}\bigr]
+\nu E\Bigl[\int_s^t \|A(u(\xi)-u(s))\|_{L^2}^2\,d\xi \Bigr]\\ \nonumber
&\quad
\leq \Bigl((2\nu+2C_0C_T)E\bigl[\|Au\|_{L^\infty(L^2)}^2\bigr]
+\frac{2}{\nu}E\bigl[\|f\|^2_{L^\infty(L^2)}\bigr] \Bigr) (1+2C_TT)e^{2C_TT}\,|t-s|.
\end{align*}
Hence, \eqref{eq2.20b} holds.

\medskip
{\em Step 3:} To show the second inequality \eqref{eq2.20c},
we apply It\^o's formula to $\Psi(u(t)):=\|A (u(t)-u(s))\|^2_{L^2}$ to get
\begin{align}\label{eq2.331}
&\|A (u(t)-u(s))\|^2_{L^2}+2\nu\int_s^t \bigl(A^{\frac{3}{2}} u(\xi)-u(s),A^{\frac{3}{2}}(u(\xi)-u(s))\bigr)\,d\xi\\\nonumber
&\quad
=2\nu\int_s^t \bigl(A^{\frac{3}{2}} u(s),A^{\frac{3}{2}}(u(\xi)-u(s))\bigr)\,d\xi
+2\int_s^t \bigl(f(\xi),A^2(u(\xi)-u(s))\bigr)\,d\xi \\ \nonumber
&\qquad\quad
+\int_s^t\|A B(\xi, u(\xi))\|^2_{\cK}\,d\xi
+2\int_s^t \bigl(A B(\xi,u(\xi))\,dW(\xi),A (u(\xi)-u(s))\bigr) \\
\nonumber
&\quad =:I_9+I_{10}+I_{11}+I_{12}.
\end{align}
We now estimate each $I_i$ for $i=9,10,11,12$ as follows,

To estimate $I_9$, we use Schwarz and Young's inequality to get
\begin{align}\label{eq2.332}
&2\nu\int_s^t \bigl(A^{\frac{3}{2}}u(s),A^{\frac{3}{2}}(u(\xi) -u(s))\bigr)\,d\xi \\ \nonumber
&\qquad \leq 2\nu \int_s^t \|A^{\frac{3}{2}} (u(\xi)-u(s))\|_{L^2} \|A^{\frac{3}{2}} u(s)\|_{L^2}\,d\xi\\ \nonumber
&\qquad \leq \frac{\nu}{2}\int_s^t \|A^{\frac{3}{2}}(u(\xi)-u(s))\|^2_{L^2}\,d\xi +2\nu\|A^{\frac{3}{2}}u(s))\|_{L^2}^2\, |t-s|.
\end{align}
Similarly, we have
\begin{align}\label{eq2.333}
2\int_s^t \bigl(f(\xi),A^{2}(u(\xi) &-u(s))\bigr)\,d\xi
\leq 2\int_s^t \|\nabla f(\xi)\|_{L^2} \|A^{\frac{3}{2}} (u(\xi)-u(s))\|_{L^2}\,d\xi \\  \nonumber
&\leq \frac{\nu}{2}\int_s^t \|A^{\frac{3}{2}}(u(\xi)-u(s))\|_{L^2}^2\,d\xi
+\frac{2}{\nu}\|f\|_{L^{\infty}(H^1)}^2\,|t-s|.
\end{align}
\begin{align}\label{eq2.334}
\int_s^t \|A B(\xi, &u(\xi))\|^2_{\cK}\,d\xi
=\int_s^t \|A B(\xi,u(\xi)) \pm A B(s,u(s))\|^2_{\cK}\,d\xi \\ \nonumber
&\leq 2C_T\int_s^t \|A(u(\xi)-u(s))\|^2_{L^2}\,d\xi
+2C_T\|\nabla^2u(s)\|^2_{L^2}\,|t-s|.
\end{align}
Here we have used the Lipschitz continuity of $B(\cdot, u)$ with $j=2$.

Since $I_{12}$ it is a martingale, then $E[I_{12}]=0$. Combining \eqref{eq2.332}--\eqref{eq2.334} and
taking the expectation we get
\begin{align*} 
&E\bigl[\|A (u(t)-u(s))\|^2_{L^2}\bigr]
+\nu E\Bigl[ \int_s^t \|A^{\frac{3}{2}}(u(\xi)-u(s))\|_{L^2}^2\,d\xi\Bigr]\\ \nonumber
&\qquad
\leq \Bigl( 2\nu E\bigl[\|A^{\frac{3}{2}}u(s)\|^2_{L^2}\bigr]+\frac{2}{\nu} E\bigl[\|f\|^2_{L^\infty(H^1)}\bigr]
+2C_TE\bigl[\|\nabla^2u(s)\|^2_{L^2}\bigr]\Bigr)\, |t-s|\\ \nonumber
&\qquad\quad
+2C_T E\Bigl[\int_s^t \|A^{\frac{3}{2}}(u(t)-u(s))\|^2_{L^2}\, d\xi \Bigr].
\end{align*}
It follows from  Gronwall inequality  that
\begin{align*} 
&E\bigl[\|A(u(t)-u(s))\|^2_{L^2}\bigr]
+\nu E\Bigl[\int_s^t \|A^{\frac{3}{2}}(u(\xi)-u(s))\|_{L^2}^2\,d\xi \Bigr]\\ \nonumber
&\quad
\leq \Bigl((2\nu+2C_0C_T)E\bigl[\|A^{\frac{3}{2}}u\|_{L^\infty(L^2)}^2\bigr]
+\frac{2}{\nu}E\bigl[\|f\|^2_{L^\infty(H^1)}\bigr] \Bigr)  (1+2C_TT)e^{2C_TT}\,|t-s|.
\end{align*}
Hence, \eqref{eq2.20c} holds. The proof is complete.
\end{proof}

\begin{remark}
Clearly, \eqref{eq2.20a} --\eqref{eq2.20c} hold under different regularity assumptions on
the solution $u$. We note that \eqref{eq2.20a}--\eqref{eq2.20c} is needed for our error analysis to be given
in the next two sections.
\end{remark}

\section{Semi-discretization in time}\label{sec-3}
In this section we analyze the Euler-Maruyama time discretization scheme for the mixed formulation
\eqref{eq2.8}--\eqref{eq2.12}. The goal is to derive optimal order error estimates in strong norms
for both the velocity and pressure approximations. The results of this section will also serve as a building
block for us to establish error estimates in strong norms for our fully discrete mixed finite element
methods to be given in the next section.

Let $N$ be a positive integer, $k:=\frac{T}{N}$ and $t_n=nk$ for $n=0,1,\cdots, N$. Set $u^0:=u_0$,
then the Euler-Maruyama scheme for \eqref{eq2.8} is defined as seeking Hilbert space valued
discrete processes $\{(u^n,p^n)\in L^2(\Omega,\cV)\times L^2(\Omega,\cW); 0\leq n\leq N\}$ such that
\begin{subequations}\label{eq3.1}
\begin{align}\label{eq3.1a}
&\bigl(u^{n+1},v\bigr) +k\,\nu a\bigl(u^{n+1}, v\bigr) +k\, b\bigl(v,p^{n+1}\bigr)\\ \nonumber
&\qquad\qquad
=(u^{n},v)+\int^{t_{n+1}}_{t_n} (f,v)\,ds + \bigl(B(t_{n},u^n)\Delta W_{n+1},v\bigr) \quad\forall v\in \cV,\,\, a.s.\\
&b\bigl(u^{n+1},q\bigr) =0 \qquad\forall q\in \cW,\,\, a.s.  \label{eq3.1b}
\end{align}
\end{subequations}
Where $\Delta W_{n+1}:=W(t_{n+1})-W(t_n)\thicksim N(0,kQ)$.

It is easy to see that \eqref{eq3.1} is a weak formulation of a random Stokes system
for $(u^{n+1}, p^{n+1})$. The well-posedness of this system immediately follows from
a generalized Lax-Milgram Theorem (also called Banach-Ne\v{c}as-Babu\v{s}ka Theorem,
cf. \cite{Ern_Guermond04}).

The next lemma establishes some stability estimates for the
discrete processes $\{(u^n,p^n)\in L^2(\Omega,\cV)\times L^2(\Omega,\cW); 0\leq n\leq N\}$.

\begin{lemma}\label{lem3.1}
The discrete processes $\{(u^n,p^n)\in L^2(\Omega,\cV)\times L^2(\Omega,\cW); 0\leq n\leq N\}$
defined by \eqref{eq3.1} satisfy
\begin{align}\label{eq3.3}
\max_{1\leq n\leq N} E\bigl[\|u^n\|^2_{L^2}\bigr] &+E\Bigl[\sum^N_{n=1}\|u^n-u^{n-1}\|^2_{L^2}\Bigr]
+\nu E\Bigl[k\sum^N_{n=1}\|\nabla u^n\|^2_{L^2}\Bigr] \\ \nonumber
&\leq C_4\Bigl\{E\bigl[\|u_0\|_{L^2}^2\bigr] +E\bigl[\|f\|^2_{L^2(0,T;H^{-1})}\bigr]\Bigr\},\\
\label{eq3.5}
E\Bigl[\Bigl\|k\sum_{n=1}^N p^n\Bigr\|_{L^2}^2\Bigr]
&\leq C_5\Bigl\{E\bigl[\|u_0\|_{L^2}^2\bigr] +E\bigl[\|f\|^2_{L^2(0,T;H^{-1})}\bigr]\Bigr\}.
\end{align}
Moreover, for $ i=1,2,3,4$
\begin{align}\label{eq3.4}
&\max_{1\leq n\leq N} E\Bigl[\|A^{\frac{i}{2}} u^n\|^2_{L^2} \Bigr]
+ E\Bigl[ \sum^N_{n=1}\|A^{\frac{i}{2}} (u^n-u^{n-1})\|^2_{L^2} \Bigr]
  +\nu  k E\Bigl[ \sum^N_{n=1}\|A^{\frac{i+1}{2}} u^n\|^2_{L^2}\Bigr] \\ \nonumber
&\hskip 1.3in \leq C_5\Bigl\{E\bigl[\|A^{\frac{i}{2}}u_0\|^2_{L^2}\bigr]
+E\bigl[\|f\|^2_{L^2(0,T;L^2)}\bigr] \Bigr\}, \\
&E\Bigl[k\sum_{n=1}^N \|\nabla p^n\|_{L^2}^2\Bigr]
\leq \frac{C_5}{k}\Bigl\{E\bigl[\|u_0\|_{L^2}^2\bigr] +E\bigl[\|f\|^2_{L^2(0,T;L^2)}\bigr] \Bigr\}
+C_5 \nu E\bigl[ \|u_0\|^2_{H^1} \bigr],
\label{eq3.4a}
\end{align}
where $C_4=C(C_0, C_T,D,T)$ and $C_5= C(C_0, C_T,D,T,\beta)$ are positive constants
	independent of $k$.
\end{lemma}

\begin{proof}
Since the proofs of \eqref{eq3.3} and \eqref{eq3.4} with $i=1$ were obtained in \cite{CHP12}, and
\eqref{eq3.4} for $i=2,3,4$ can be derived similarly by taking $v=A^iu^{n+1}$ in \eqref{eq3.1a},  below
we only give proofs for \eqref{eq3.5} and \eqref{eq3.4a}.

Applying the summation operator $\sum_{n=1}^N$ to both sides of \eqref{eq3.1a}  (after lowering
the super-index by one) we get
 \begin{align*}
b\Bigl( v, k\sum_{n=1}^N p^{n}\Bigr) &= \bigl(u^{0},v\bigr) - (u^N,v) +\int^{T}_{0} (f,v)\,ds
\nonumber \\
&\quad
-k\,\nu \sum_{n=1}^N a\bigl(u^{n}, v\bigr) + \sum_{n=1}^N \bigl(B(t_{n-1},u^{n-1})\Delta W_{n},v\bigr) \quad\forall v\in \cV,\,\, a.s\\
\end{align*}
Taking the expectation and using Schwarz inequality and Poinc\'{a}re-Freidrich's inequality on the right-hand side we obtain
\begin{align*}
E\Bigl[b\Bigl( v, k\sum_{n=1}^N p^{n}\Bigr)\Bigr]
&\leq C_0\Bigl( E\bigl[\|u^{0}\|_{L^2}^2 +\|u^{N}\|_{L^2}^2 +\|f\|_{L^1(0,T;H^{-1})}^2\bigr]\Bigr)^\frac12 \bigl(E\bigl[\|\nabla v\|_{L^2}^2\bigr]\bigr)^\frac12 \\
&\quad
+\nu \Bigl(E\Bigl[k\sum_{n=1}^N \|\nabla u^{n}\|^2_{L^2}\Bigr]\Bigr)^{\frac{1}{2}}
\bigl(E\bigl[\|\nabla v\|^2_{L^2}\bigr]\bigr)^{\frac{1}{2}}\\
&\quad
+\Bigl(E\Bigl[\Bigl\|\sum_{n=1}^N \int_{t_{n-1}}^{t_n} B(t_{n-1},u^{n-1})\, dW(s)\Bigr\|_{L^2}^2\Bigr] \Bigr)^\frac12 \bigl(E\bigl[\|v\|_{L^2}^2\bigr]\Bigr)^\frac12.
\end{align*}
Then by the inf-sup condition, It\^{o}'s isometry and \eqref{eq3.3}, we get
\begin{align*}
&\beta \Bigl\| k\,\sum_{n=1}^{N} p^n \Bigr\|_{\mW}
\leq C_0\Bigl( E\bigl[\|u^{0}\|_{L^2}^2 +\|u^{N}\|_{L^2}^2 +\|f\|_{L^1(0,T;H^{-1})}^2\bigr]\Bigr)^\frac12\\
&\qquad\qquad
+\nu \Bigl(E\Bigl[k\sum_{n=1}^N \|\nabla u^{n}\|^2_{L^2}\Bigr]\Bigr)^{\frac{1}{2}}
+C_0\Bigl(E\Bigl[\sum_{n=1}^N \int_{t_{n-1}}^{t_n} \bigl\|B(t_{n-1},u^{n-1})\bigr\|_{\cK}^2\, ds\Bigr] \Bigr)^\frac12\\
&\quad
\leq C \Bigl\{ E\bigl[\|u_0\|^2_{L^2}\bigr] +E\bigl[\|f\|^2_{L^2(0,T;H^{-1})}\bigr] \Bigr\}^{\frac{1}{2}}
 +C_0 \Bigl(E\Bigl[k\sum_{n=1}^N C^2_T \bigl(1+\|u^{n-1}\|_{L^2}\bigr)^2 \Bigr] \Bigr)^{\frac{1}{2}}\\
&\quad
\leq  \beta\sqrt{C_5}E\Bigl\{\bigl[\|u_0\|_{L^2}^2\bigr]
	+E\bigl[\|f\|^2_{L^2(0,T;H^{-1})}\bigr]\Bigr\}^\frac12. 
\end{align*}
 Hence, \eqref{eq3.5} holds.

To prove \eqref{eq3.4a}, we first notice that \eqref{eq3.1a} can be rewritten as
\begin{align}
  & k\bigl( \div v, p^n \bigr) = k\, b\bigl(v,p^{n}\bigr) = \bigl\langle G^n, v\bigr\rangle   \qquad \forall v\in \mV \quad \mbox{a.s.}, \label{eq3.100}
\end{align}
where
\begin{align*}
 G^n  := u^{n-1}-u^n -\nu k Au^n + f^n + B(t_{n-1}, u^{n-1}) \Delta W_n, \quad 
 f^n :=\int_{t_{n-1}}^{t_n} f(s)\,ds.
\end{align*}
It can be shown that $u^n$ and $Au^{n}$ are $\cF_{t_{n}}$-measurable as done in \cite{DD2004}.
By the assumptions on $u^n$ and $f$, it is easy to check that $G^n\in L^2(\Omega, L^2(D))$
and $ \bigl\langle G^n, v\bigr\rangle =0$ for all $v\in L^2(\Omega,\cV_0)$.  Then it follows from
Propositions 1.1 and 1.2 of \cite{Temam01} that $G^n$ has a (scalar) potential $p^n\in L^2(\Omega, H^1(D))$,  that is,
 $p^n$ is a solution of \eqref{eq3.100}.  Moreover, there exists some positive constant $C$ such that
\begin{equation} \label{eq3.200}
k \|\nabla p^n\|_{L^2(\Omega,L^2)} \leq C \|G^n\|_{L^2(\Omega,L^2)},  \qquad n=1,2,\cdots, N.
\end{equation}

It now remains to bound the right-hand side of \eqref{eq3.200}.    To this end, first
using the triangle inequity and It\^{o}'s  isometry we get
\begin{align*}
\|G^n\|^2_{L^2(\Omega,L^2)} & \leq 2E\bigl[ \|u^{n}-u^{n-1}\|^2_{L^2} \bigr]
+ 2k^2 \nu^2E\bigl[ \| Au^{n}\|^2_{L^2}\bigr] \\
&\qquad
+ 2E\bigl[ \|f^n\|^2_{L^2}\bigr]
+ 2 E\Bigl[ \Bigl\| \int_{t_{n-1}}^{t_n} B(t_{n-1},u^{n-1})\, d W(s)  \Bigr\|^2_{L^2} \Bigr] \\
& \leq 2E\bigl[ \|u^{n}-u^{n-1}\|^2_{L^2} \bigr]
+ 2k^2\nu^2 E\bigl[ \| Au^{n}\|^2_{L^2}\bigr] \\
&\qquad
+ 2E\bigl[ \|f^n\|^2_{L^2}\bigr]
+ 4k C_T^2 E\bigl[ 1+\| u^{n-1}\|^2_{L^2} \bigr].
\end{align*}
Then applying the operator  $\sum_{n=1}^N$ to \eqref{eq3.200}  (after squaring it),
using the above inequality, \eqref{eq3.3} and \eqref{eq3.4} we obtain
\begin{align*}
 k\, E\Bigl[ k\sum_{n=1}^N\|\nabla p^{n}\|^2_{L^2} \bigr]
&\leq 2C^2 E\Bigl[\sum_{n=1}^N\|u^{n}-u^{n-1}\|^2_{L^2}\Bigr]
+ 2C^2\nu^2 k E\Bigl[ k\sum_{n=1}^N\| Au^{n}\|^2_{L^2} \Bigr] \\
&\quad+ 2C^2 E\Bigl[ \sum_{n=1}^N \|f^n\|^2_{L^2}\Bigr]
+4C^2 C_T^2 E\Bigl[ k\sum_{n=1}^N \bigl( 1+ \|u^{n-1}\|^2_{L^2}\bigr) \Bigr]\\
&\leq C_5\Bigl\{E\bigl[\|u_0\|^2_{L^2}\bigr]+E\bigl[ \|f\|^2_{L^2(0,T;L^2)} \bigr]
+\nu k E\bigl[ \|u_0\|^2_{H^1} \bigr]  \Bigr\}.
\end{align*}
Hence,  \eqref{eq3.4a}  holds and the proof is complete.
\end{proof}

\begin{remark}
	(a) We note that \eqref{eq3.5}  and \eqref{eq3.4a} are different bounds for the pressure.
	While the time average estimate  \eqref{eq3.5} allows for a uniform bound in the discretization parameter $k>0$,
	the non-uniform estimate \eqref{eq3.4a} reflects the subtle interplay of the
	Lagrange multiplier $p$ with the (non-solenoidal) noise on the right-hand side of \eqref{eq3.1a}.
	
	(b) We note that stability estimates similar to \eqref{eq3.3} and \eqref{eq3.4} {with $i=1$  }
	were also established for the projection method in \cite{CHP12}.
    Here we extend these results to the mixed method.
	
	(c)	We emphasize that the stability estimates \eqref{eq3.4} and \eqref{eq3.4a} will be crucially
	used in the error analysis for our fully discrete finite element method in Section \ref{sec-4}.

    (d) We note that higher regularities of $p^n$ are required if one wants to obtain higher
    order space error estimates, which can be proved in the similar fashion.
\end{remark}

The first main result of this section is stated in the following theorem.

\begin{theorem}\label{thm3.2}
The time discrete velocity process $\{u^n; 0\leq n\leq N\}$ defined by scheme \eqref{eq3.1} satisfies the
following error estimate:
\begin{align}\label{eq3.6}
&\max_{1\leq n\leq N} E\bigl[\|u(t_n)-u^n\|^2_{L^2}\bigr]
+\nu E\Bigl[k\sum_{n=1}^{N} \|\nabla (u(t_n)-u^n)\|^2_{L^2}\Bigr]\leq C_{6} k,\\\label{eq3.6a}&
\max_{1\leq n\leq N} E\bigl[\|\nabla (u(t_n)-u^n)\|^2_{L^2}\bigr]
+\nu E\Bigl[k\sum_{n=1}^{N} \|A(u(t_n)-u^n)\|^2_{L^2}\Bigr]\leq C_{7} k,\\\label{eq3.6b}&
\max_{1\leq n\leq N} E\bigl[\|A(u(t_n)-u^n)\|^2_{L^2}\bigr]
+\nu E\Bigl[k\sum_{n=1}^{N} \|A^{\frac{3}{2}}(u(t_n)-u^n)\|^2_{L^2}\Bigr]\leq C_{8} k
\end{align}
for some positive constants $C_6=C(C_T,C_1,f,D,T)$,
$C_7=C(C_T,C_2,f,D,T)$, $C_{8}=C(C_T,C_3,f,D,T)$ independent of $k$.
\end{theorem}

\begin{proof}
We first like point out that since at each time step $t_n$, the velocity field $u^n$ satisfy the
divergence-free condition $\div u^n =0$
$\mP$-a.s.,  restricting the test function $v\in \cV_0$ in \eqref{eq3.1a} then eliminates the pressure term.
The desired estimate \eqref{eq3.6} follows from a similar estimate of  \cite{CP12}.
Below we give a proof for each of \eqref{eq3.6}--\eqref{eq3.6b} for the sake of completeness.
	
It follows from \eqref{eq2.8} that the velocity field $u$ satisfies $\mP$-a.s.
\begin{align}\label{eq3.7}
\big( u(t_{n+1}),v\bigr) &-\bigl(u(t_n),v\bigr) +\nu\int_{t_n}^{t_{n+1}} a\bigl(u(s),v\bigr)\,ds\\
\nonumber
&=\int_{t_n}^{t_{n+1}} \bigl(f(s),v\bigr) \,ds +\Bigl(\int_{t_n}^{t_{n+1}} B(s,u(s))\,dW(s),v\Bigr)
\quad\forall  v\in \cV_0.
\end{align}
Let $e^n:=u(t_n)-u^n$, subtracting \eqref{eq3.1a} from \eqref{eq3.7} yields
\begin{align}\label{eq3.8}
\bigl(e^{n+1},v\bigr)-\bigl(e^{n},v\bigr) &+\nu\int_{t_n}^{t_{n+1}} a\bigl(u(s)-u^{n+1}, v\bigr)\,ds\\
&=\Bigl(\int_{t_n}^{t_{n+1}}\bigl( B(s,u(s))-B(t_n,u^n)\bigr)\, dW(s),v\Bigr)
\qquad\forall  v\in \cV_0. \nonumber
\end{align}
Note that $e^{n+1}\in \cV_0$. Choosing $v=e^{n+1}$ in \eqref{eq3.8} and using the identity
$(a,a-b)=\frac{1}{2}(\|a\|^2-\|b\|^2+\|a-b\|^2)$ (for any two $d$-vectors $a$ and $b$) we get
\begin{align}\label{eq3.9}
\frac{1}{2} \Bigl[\|e^{n+1}\|^2_{L^2} &-\|e^{n}\|^2_{L^2} +\|e^{n+1}-e^{n}\|^2_{L^2} \Bigr]
+\nu k\|\nabla e^{n+1}\|^2_{L^2}\\ \nonumber
&=\nu\int_{t_n}^{t_{n+1}} \bigl(\nabla (u(t_{n+1})- u(s)), \nabla e^{n+1}\bigr)\,ds \\ \nonumber
&\qquad\quad
+\Bigl(\int_{t_n}^{t_{n+1}}\bigl(B(s,u(s))-B(t_n,u^n)\bigr)\,dW(s),e^{n+1}\Bigr).
\end{align}

We now bound the first term on the right-hand using \eqref{eq3.9} as follows:
\begin{align}\label{eq3.10}
\nu\int_{t_n}^{t_{n+1}} &\bigl(\nabla (u(s)- u(t_{n+1})),\nabla e^{n+1}\bigr)\,ds\\
&\leq \frac{k\nu}{4}\|\nabla e^{n+1}\|^2_{L^2}
+\nu\int_{t_n}^{t_{n+1}} \|\nabla (u(s)- u(t_{n+1}))\|^2_{L^2}\,ds \nonumber\\
&\leq\frac{\nu k}{4}\|\nabla e^{n+1}\|^2_{L^2}+C_1k.  \nonumber
\end{align}
Substituting \eqref{eq3.10} into \eqref{eq3.9}, summing over $n$ and taking
the expectation we get
\begin{align}\label{eq3.11}
&\frac12 \max_{0\leq n\leq N-1} E\bigl[\|e^{n+1}\|^2_{L^2}\bigr]
+\frac12\sum_{n=0}^{N-1} E\Bigl[\|e^{n+1}-e^{n}\|^2_{L^2}
+2\nu k\|\nabla e^{n+1}\|^2_{L^2}\Bigr] \\ \nonumber
&\qquad
\leq \sum_{n=0}^{N-1} E\Bigl[\Bigl(\int_{t_n}^{t_{n+1}} \bigl(B(s,u(s))-B(t_n,u^n)\bigr)\,dW(s) ,e^{n+1}\Bigr)\Bigr]
+ C_1k.
\end{align}

It remains to bound the the first term on the right-hand of \eqref{eq3.11}. To this end, we write
\begin{align}\label{eq3.13}
&\sum^{N-1}_{n=0} E\Bigl[ \Bigl(\int_{t_n}^{t_{n+1}} \bigl(B(s,u(s))-B(t_n,u^n)\bigr)\,dW(s),e^{n+1}\Bigr)\Bigr]\\
\nonumber
&\qquad
=E\Bigl[\,\sum^{N-1}_{n=0} \Bigl(\int_{t_n}^{t_{n+1}} \bigl(B(s,u(s))-B(t_n,u^n)\bigr)\,dW(s),e^{n+1}-e^n\bigl)\Bigr)\Bigr]\\ \nonumber
&\qquad\qquad
+E\Bigl[\,\sum^{N-1}_{n=0} \Bigl(\int_{t_n}^{t_{n+1}} \bigl(B(s,u(s))-B(t_n,u^n)\bigr)\,dW(s),e^{n}\bigr)\Bigr)\Bigr]
=:I+II.
\end{align}
By Schwarz inequality, \eqref{ito_int}, \eqref{eq2.2} and Theorem \ref{thm2.2} we get
\begin{align}\label{eq3.14}
|I|&\leq E\Bigl[\,\sum^{N-1}_{n=0} \Bigl|\Bigl(\int_{t_n}^{t_{n+1}}\bigl(B(s,u(s))-B(t_n,u^n)\bigr)\,dW(s),e^{n+1}-e^n\Bigr)\Bigr|\Bigr]\\\nonumber
&\leq \sum^{N-1}_{n=0} E\Bigl[\Bigl\|\int_{t_n}^{t_{n+1}} \bigl(B(s,u(s))-B(t_n,u^n)\bigr)\,dW(s)\Bigr\|_{L^2}\|e^{n+1}-e^n\|_{L^2}\Bigr]\\ \nonumber
&\leq \sum^{N-1}_{n=0} E\Bigl[ \int_{t_n}^{t_{n+1}} \|B(s,u(s))-B(t_n,u^n)\|_{\cK}\,ds\,
\|e^{n+1}-e^n\|_{L^2}\Bigr]\\ \nonumber
&\leq \sum^{N-1}_{n=0} E\Bigl[\int_{t_n}^{t_{n+1}} C_T\bigl(k^{\frac12} + \|u(s)- u^n\|_{L^2}\bigr) \, ds\,
\|e^{n+1}-e^n\|_{L^2}\Bigr]\\ \nonumber
&\leq \sum^{N-1}_{n=0} E\Bigl[\int_{t_n}^{t_{n+1}} C_T\bigl(k^{\frac12} +\|u(s)- u(t_n)\|_{L^2}
+\|e^n\|_{L^2} \bigr)\, ds\,\|e^{n+1}-e^n\|_{L^2}\Bigr]\\ \nonumber
&\leq \sum^{N-1}_{n=0} \Bigl[c_3 \Bigl(k^{\frac52} + k^2E\bigl[\|e^n\|^2_{L^2}\bigr]\Bigr)
+\frac14 \|e^{n+1}-e^n\|^2_{L^2}\Bigr]\\ \nonumber
&\leq c_3\Bigl(T k^{\frac32} + k^2\sum^{N-1}_{n=0} E\bigl[\|e^{n}\|^2_{L^2}\bigr]\Bigr)
+\frac14 \sum^{N-1}_{n=0} E\bigl[\|e^{n+1}-e^n\|^2_{L^2}\bigr].
\end{align}
Where $c_3>0$ is a constant which is independent of $k$.

On the other hand,  by a well-known property of martingales we have $II=0$.
Combining \eqref{eq3.11}--\eqref{eq3.14} we get
\begin{align}\label{eq3.18}
\max_{0\leq n\leq N-1} E\bigl[\|e^{n+1}\|^2_{L^2}\bigr]
+\nu k\sum_{n=0}^{N-1} E\bigr[\|\nabla e^{n+1}\|^2_{L^2}\bigr]
\leq c_4 k\sum_{n=0}^{N-1} E\bigl[\|e^n\|^2_{L^2}\bigr] +c_5 k,
\end{align}
where $c_4= 4Tc_3 k$ and $c_5= 4 C_1 + 4T c_3 k^{\frac12}$.

Appying Gronwall's inequality in \eqref{eq3.18} yields
\begin{align}\label{eq3.19}
\max_{0\leq n\leq N-1} E\bigl[\|e^{n+1}\|^2_{L^2}\bigr]
+\nu k\sum_{n=0}^{N-1} E\bigr[\|\nabla e^{n+1}\|^2_{L^2}\bigr]
\leq c_5\exp(c_4T)\, k.
\end{align}
Then, \eqref{eq3.6a} holds with $C_4=c_5 \exp(c_4T)$.

Noting that $Ae^{n+1}\in \cV_0$, and setting $v=Ae^{n+1}$ in \eqref{eq3.8} yields
\begin{align}\label{eq3.191}
\frac{1}{2} \Bigl[\|\nabla e^{n+1}\|^2_{L^2} &-\|\nabla e^{n}\|^2_{L^2} +\|\nabla (e^{n+1}-e^{n})\|^2_{L^2} \Bigr]
+\nu k\|Ae^{n+1}\|^2_{L^2}\\ \nonumber
&=\nu\int_{t_n}^{t_{n+1}} \bigl(A(u(t_{n+1})- u(s)), A e^{n+1}\bigr)\,ds \\ \nonumber
&\qquad\quad
+\Bigl(\int_{t_n}^{t_{n+1}}\bigl(B(s,u(s))-B(t_n,u^n)\bigr)\,dW(s), Ae^{n+1}\Bigr).
\end{align}
Using \eqref{eq2.20b}, the first term of right hand side in \eqref{eq3.191} can be bounded by
\begin{align}\label{eq3.192}
\nu\int_{t_n}^{t_{n+1}} &\bigl(A (u(s)- u(t_{n+1})),A e^{n+1}\bigr)\,ds\\
&\leq \frac{k\nu}{4}\|A e^{n+1}\|^2_{L^2}
+\nu\int_{t_n}^{t_{n+1}} \|A (u(s)- u(t_{n+1}))\|^2_{L^2}\,ds \nonumber\\
&\leq\frac{\nu k}{4}\|A e^{n+1}\|^2_{L^2}+C_2k.  \nonumber
\end{align}
Substituting \eqref{eq3.192} into \eqref{eq3.191}, summing over $n$ and taking
the expectation we have
\begin{align}\label{eq3.193}
&\frac12 \max_{0\leq n\leq N-1} E\bigl[\|\nabla e^{n+1}\|^2_{L^2}\bigr]
+\frac12\sum_{n=0}^{N-1} E\Bigl[\nabla(\|e^{n+1}-e^{n})\|^2_{L^2}
+2\nu k\|A e^{n+1}\|^2_{L^2}\Bigr] \\ \nonumber
&\qquad
\leq \sum_{n=0}^{N-1} E\Bigl[\Bigl(\int_{t_n}^{t_{n+1}} \bigl(\nabla (B(s,u(s))-B(t_n,u^n))\bigr)\,dW(s), \nabla e^{n+1}\Bigr)\Bigr]
+ C_2k.
\end{align}
It remains to bound the the first term on the right-hand of \eqref{eq3.193}. To this end, we rewrite as follows
\begin{align}\label{eq3.194}
&\sum^{N-1}_{n=0} E\Bigl[ \Bigl(\int_{t_n}^{t_{n+1}} \bigl(\nabla (B(s,u(s))-B(t_n,u^n))\bigr)\,dW(s),\nabla e^{n+1}\Bigr)\Bigr]\\
\nonumber
&\quad
=E\Bigl[\,\sum^{N-1}_{n=0} \Bigl(\int_{t_n}^{t_{n+1}} \bigl(\nabla (B(s,u(s))-B(t_n,u^n))\bigr)\,dW(s),\nabla e^{n+1}-\nabla e^n\bigl)\Bigr)\Bigr]\\ \nonumber
&\qquad\qquad
+E\Bigl[\,\sum^{N-1}_{n=0} \Bigl(\int_{t_n}^{t_{n+1}} \bigl(\nabla (B(s,u(s))-B(t_n,u^n))\bigr)\,dW(s),\nabla e^{n}\bigr)\Bigr)\Bigr]\\ \nonumber
&\quad =:III+IV.
\end{align}
By Schwarz inequality, \eqref{ito_int}, \eqref{eq2.2} and Theorem \ref{thm2.2} we get
\begin{align}\label{eq3.195}
&|III|\leq E\Bigl[\,\sum^{N-1}_{n=0} \Bigl|\Bigl(\int_{t_n}^{t_{n+1}}\bigl(\nabla (B(s,u(s))-B(t_n,u^n))\bigr)\,dW(s),\nabla (e^{n+1}-e^n)\Bigr)\Bigr|\Bigr]\\\nonumber
&\leq \sum^{N-1}_{n=0} E\Bigl[\Bigl\|\int_{t_n}^{t_{n+1}} \bigl(\nabla (B(s,u(s))-B(t_n,u^n))\bigr)\,dW(s)\Bigr\|_{L^2}\|\nabla (e^{n+1}-e^n)\|_{L^2}\Bigr]\\ \nonumber
&\leq \sum^{N-1}_{n=0} E\Bigl[ \int_{t_n}^{t_{n+1}} \|\nabla (B(s,u(s))-B(t_n,u^n))\|_{\cK}\,ds\,
\|\nabla (e^{n+1}-e^n)\|_{L^2}\Bigr]\\ \nonumber
&\leq \sum^{N-1}_{n=0} E\Bigl[\int_{t_n}^{t_{n+1}} C_T\bigl(k^{\frac12} +\|\nabla (u(s)- u(t_n))\|_{L^2}
+\|\nabla e^n\|_{L^2} \bigr)\, ds  \|\nabla (e^{n+1}-e^n)\|_{L^2}\Bigr]\\ \nonumber
&\leq \sum^{N-1}_{n=0} \Bigl[c_6 \Bigl(k^{\frac52} + k^2E\bigl[\|\nabla e^n\|^2_{L^2}\bigr]\Bigr)
+\frac14 \|\nabla (e^{n+1}-e^n)\|^2_{L^2}\Bigr]\\ \nonumber
&\leq c_6\Bigl(T k^{\frac32} + k^2\sum^{N-1}_{n=0} E\bigl[\|\nabla e^{n}\|^2_{L^2}\bigr]\Bigr)
+\frac14 \sum^{N-1}_{n=0} E\bigl[\|\nabla (e^{n+1}-e^n)\|^2_{L^2}\bigr].
\end{align}

By a well-known property of martingales we obtain $IV=0$.
Combining \eqref{eq3.193}--\eqref{eq3.195} we have
\begin{align}\label{eq3.196}
\max_{0\leq n\leq N-1} E\bigl[\|\nabla e^{n+1}\|^2_{L^2}\bigr]
&+\nu k\sum_{n=0}^{N-1} E\bigr[\|A e^{n+1}\|^2_{L^2}\bigr] \\
&\leq c_7 k\sum_{n=0}^{N-1} E\bigl[\|\nabla e^n\|^2_{L^2}\bigr] +c_8 k, \nonumber
\end{align}
where $c_7= 4Tc_6 k$ and $c_8= 4 C_2 + 4T c_6 k^{\frac12}$.

Applying the Gronwall's inequality to \eqref{eq3.196} yields
\begin{align}\label{eq3.197}
\max_{0\leq n\leq N-1} E\bigl[\|e^{n+1}\|^2_{L^2}\bigr]
+\nu k\sum_{n=0}^{N-1} E\bigr[\|\nabla e^{n+1}\|^2_{L^2}\bigr]
\leq c_8\exp(c_7T)\, k.
\end{align}
Hence, \eqref{eq3.6a} holds with $C_7=c_8 \exp(c_7T)$.

To prove \eqref{eq3.6b}, since $A^2e^{n+1}\in \cV_0$, we can set $v=A^2e^{n+1}$ in \eqref{eq3.8} and get
\begin{align}\label{eq3.198}
\frac{1}{2} \Bigl[\|A e^{n+1}\|^2_{L^2} &-\|A e^{n}\|^2_{L^2} +\|A (e^{n+1}-e^{n})\|^2_{L^2} \Bigr]
+\nu k\|A^{\frac{3}{2}}e^{n+1}\|^2_{L^2}\\ \nonumber
&=\nu\int_{t_n}^{t_{n+1}} \bigl(A^{\frac{3}{2}}(u(t_{n+1})- u(s)), A^{\frac{3}{2}} e^{n+1}\bigr)\,ds \\ \nonumber
&\qquad
+\Bigl(\int_{t_n}^{t_{n+1}}\bigl(A(B(s,u(s))-B(t_n,u^n))\bigr)\,dW(s), Ae^{n+1}\Bigr).
\end{align}
 Using \eqref{eq2.20c}, the first term on the right-hand side of \eqref{eq3.198} can be bounded as
\begin{align}\label{eq3.199}
\nu\int_{t_n}^{t_{n+1}} &\bigl(A^{\frac{3}{2}} (u(s)- u(t_{n+1})),A^{\frac{3}{2}} e^{n+1}\bigr)\,ds\\
&\leq \frac{k\nu}{4}\|A^{\frac{3}{2}} e^{n+1}\|^2_{L^2}
+\nu\int_{t_n}^{t_{n+1}} \|A^{\frac{3}{2}} (u(s)- u(t_{n+1}))\|^2_{L^2}\,ds \nonumber\\
&\leq\frac{\nu k}{4}\|A^{\frac{3}{2}} e^{n+1}\|^2_{L^2}+C_3k.  \nonumber
\end{align}
Substituting \eqref{eq3.199} into \eqref{eq3.198}, summing over $n$ and taking
the expectation we get
\begin{align}\label{eq3.20}
&\frac12 \max_{0\leq n\leq N-1} E\bigl[\|A e^{n+1}\|^2_{L^2}\bigr]
+\frac12\sum_{n=0}^{N-1} E\Bigl[A(\|e^{n+1}-e^{n})\|^2_{L^2}
+2\nu k\|A^{\frac{3}{2}} e^{n+1}\|^2_{L^2}\Bigr] \\ \nonumber
&\qquad
\leq \sum_{n=0}^{N-1} E\Bigl[\Bigl(\int_{t_n}^{t_{n+1}} \bigl(A (B(s,u(s))-B(t_n,u^n))\bigr)\,dW(s), A e^{n+1}\Bigr)\Bigr]
+ C_3k.
\end{align}

It remains to bound the the first term on the right-hand of \eqref{eq3.20}. To this end, we write
\begin{align}\label{eq3.20a}
&\sum^{N-1}_{n=0} E\Bigl[ \Bigl(\int_{t_n}^{t_{n+1}} \bigl(A(B(s,u(s))-B(t_n,u^n))\bigr)\,dW(s),A e^{n+1}\Bigr)\Bigr]\\
\nonumber
&\quad
=E\Bigl[\,\sum^{N-1}_{n=0} \Bigl(\int_{t_n}^{t_{n+1}} \bigl(A (B(s,u(s))-B(t_n,u^n))\bigr)\,dW(s),A e^{n+1}-A e^n\bigl)\Bigr)\Bigr]\\ \nonumber
&\qquad\qquad
+E\Bigl[\,\sum^{N-1}_{n=0} \Bigl(\int_{t_n}^{t_{n+1}} \bigl(A (B(s,u(s))-B(t_n,u^n))\bigr)\,dW(s),A e^{n}\bigr)\Bigr)\Bigr] \\ \nonumber
&\quad =:V+VI.
\end{align}
By Schwarz inequality, \eqref{ito_int}, \eqref{eq2.2} and Theorem \ref{thm2.2} we get
\begin{align}\label{eq3.20b}
|V|&\leq E\Bigl[\,\sum^{N-1}_{n=0} \Bigl|\Bigl(\int_{t_n}^{t_{n+1}}\bigl(A (B(s,u(s))-B(t_n,u^n))\bigr)\,dW(s),A (e^{n+1}-e^n)\Bigr)\Bigr|\Bigr]\\\nonumber
&\leq \sum^{N-1}_{n=0} E\Bigl[\Bigl\|\int_{t_n}^{t_{n+1}} \bigl(A (B(s,u(s))-B(t_n,u^n))\bigr)\,dW(s)\Bigr\|_{L^2}\|A (e^{n+1}-e^n)\|_{L^2}\Bigr]\\ \nonumber
&\leq \sum^{N-1}_{n=0} E\Bigl[ \int_{t_n}^{t_{n+1}} \|A(B(s,u(s))-B(t_n,u^n))\|_{\cK}\,ds\,
\|A (e^{n+1}-e^n)\|_{L^2}\Bigr]\\ \nonumber
&\leq \sum^{N-1}_{n=0} E\Bigl[\int_{t_n}^{t_{n+1}} C_T\bigl(k^{\frac12} +\|A (u(s)- u(t_n))\|_{L^2}
+\|Ae^n\|_{L^2} \bigr)\, ds\,\|A(e^{n+1}-e^n)\|_{L^2}\Bigr]\\ \nonumber
&\leq \sum^{N-1}_{n=0} \Bigl[c_9 \Bigl(k^{\frac52} + k^2E\bigl[\|A e^n\|^2_{L^2}\bigr]\Bigr)
+\frac14 \|A (e^{n+1}-e^n)\|^2_{L^2}\Bigr]\\ \nonumber
&\leq c_9\Bigl(T k^{\frac32} + k^2\sum^{N-1}_{n=0} E\bigl[\|A e^{n}\|^2_{L^2}\bigr]\Bigr)
+\frac14 \sum^{N-1}_{n=0} E\bigl[\|A (e^{n+1}-e^n)\|^2_{L^2}\bigr].
\end{align}
By a well-known property of martingales we get $VI=0$. Combining \eqref{eq3.20}--\eqref{eq3.20b} yields
\begin{align}\label{eq3.20c}
\max_{0\leq n\leq N-1} E\bigl[\|\nabla e^{n+1}\|^2_{L^2}\bigr]
&+\nu k\sum_{n=0}^{N-1} E\bigr[\|A e^{n+1}\|^2_{L^2}\bigr] \\
&\leq c_{10} k\sum_{n=0}^{N-1} E\bigl[\|\nabla e^n\|^2_{L^2}\bigr] +c_{11} k,  \nonumber
\end{align}
where $c_{10}= 4Tc_9 k$ and $c_{11}= 4 C_3 + 4T c_9 k^{\frac12}$.

Finally, it follows from \eqref{eq3.20c} and Gronwall's inequality that
\begin{align}\label{eq3.20d}
\max_{0\leq n\leq N-1} E\bigl[\|Ae^{n+1}\|^2_{L^2}\bigr]
+\nu k\sum_{n=0}^{N-1} E\bigr[\|A^{\frac{3}{2}} e^{n+1}\|^2_{L^2}\bigr]
\leq c_{11}\exp(c_{10}T)\, k.
\end{align}
The proof is completed after setting $C_8=c_{11}\exp(c_{10}T)$.
\end{proof}

Next, we state the second main result of this section which gives an error estimate for the
pressure approximation, such an estimate  has not been known before.

\begin{theorem}\label{thm3.3}
Let $\{p^n; 1\leq n\leq N\}$ be the pressure approximation defined by scheme \eqref{eq3.1}.
Then there holds for $m=1,2,\cdots, N$
\begin{align}\label{eq3.21}
E\Bigl[\Bigl\|\int_0^{t_m}p(s)\,ds-k\sum^m_{n=1}p^n \Bigr\|_{L^2} \Bigr]\leq C_9 k^{\frac12}.
\end{align}
where $C_9=\beta^{-1} (\nu C_1+C_6)^{\frac12} T^{\frac12}$.

\end{theorem}

\begin{proof}
The proof is based on the inf-sup property (see Lemma \ref{lem2.1}) and the error estimate
for the velocity approximation established in the previous theorem. To the end, summing
\eqref{eq3.1a} (after lowering the index by one) over $1\leq n\leq m (\leq N)$ we get
\begin{align}\label{eq3.21a}
&\bigl(u^m,v\bigr) +\nu k\sum_{n=1}^m a\bigl(u^n, v\bigr) +k \sum_{n=1}^m b\bigl(v,p^n\bigr)\\ \nonumber
&\qquad
=(u^0,v)+\int^{t_m}_{0} (f,v)\,ds + \sum_{n=1}^m \bigl(B(t_{n-1},u^{n-1}) \Delta W_n,v\bigr)
\quad\forall v\in \cV,\,\, a.s.
\end{align}
Subtracting \eqref{eq2.8a} (with $t=t_m$) from \eqref{eq3.21a} and noting that $u^0=u(0)$ we get
\begin{align}\label{eq3.22}
&b\Bigl(v,k\sum^m_{n=1}p^n -\int_0^{t_m}p(s)\, ds \Bigr)
=\nu \sum^{m-1}_{n=0} \int_{t_n}^{t_{n+1}} a\bigl(u(s)-u^n, v)\, ds\\ \nonumber
&\qquad +\sum^m_{n=1} \Bigl(\int^{t_n}_{t_{n-1}} \bigl(B(t_{n-1},u^{n-1})-B(s,u(s))\bigr)\,dW(s),v\Bigr)
+\bigl(u(t_m)-u^m,v \bigr) \\ \nonumber
&\quad = \nu\sum^{m-1}_{n=0}\int_{t_n}^{t_{n+1}}\bigl(\nabla u(s)-\nabla u(t_n)+  \nabla u(t_n)-\nabla u^n,\nabla v\bigr)\, ds\\
\nonumber
&\qquad +\sum^m_{n=1} \Bigl(\int^{t_n}_{t_{n-1}}\bigl(B(t_{n-1},u^{n-1})-B(s,u(s))\bigr)\,dW(s),v \Bigr)
+\bigl(u(t_m)-u^m,v \bigr).
\end{align}
 Taking the expectation on both sides of \eqref{eq3.22} and using
 the Poincar\'{e} inequality and \eqref{eq3.6} we obtain
\begin{align}\label{eq3.23}
E\Bigl[b\Bigl(v,k\sum^m_{n=1}p^n &-\int_0^{t_m}p(s)\,ds \Bigr)\Bigr]
\leq (\nu C_1+C_4)^{\frac12} T^{\frac12}k^{\frac12} \bigl(E\bigl[\|\nabla v\|^2_{L^2}\bigr]\bigr)^{\frac12} \\
\nonumber
&+E\Bigl[\Bigl(\sum^m_{n=1}\int^{t_n}_{t_{n-1}}\bigl(B(t_{n-1},u^{n-1})-B(s,u(s))\bigr)\,dW(s),v\Bigr)\Bigr]\\
\nonumber
&\hskip 1.05in
\leq (\nu C_1+ C_6)^{\frac12} T^{\frac12}k^{\frac12} \bigl(E\bigl[\|\nabla v\|^2_{L^2}\bigr]\bigr)^{\frac12},
\end{align}
here we have used a well-known property of martingales to conclude that the stochastic integral term vanishes.

Finally, it follows from \eqref{eq3.23} and the inf-sup condition \eqref{eq2.15} that
\begin{align}\label{eq3.25}
\beta \Bigl\|\int_0^{t_m} p(s)\,ds -k\sum^m_{n=1}p^n \Bigr\|_{\mW}
\leq (\nu C_1+C_6)^{\frac12} T^{\frac12}\,k^{\frac12},
\end{align}
which infers the desired estimate \eqref{eq3.21}.  The proof is complete.
\end{proof}

\begin{remark}
We remark that $O(k^{\frac12})$ order error estimate is optimal for the Euler-Maruyama scheme,
hence both estimates \eqref{eq3.6}--\eqref{eq3.6b}and \eqref{eq3.21} are optimal.
 On the other hand, we note that the norm used to measure the
pressure approximation error is a weaker norm compared to the norm
which is often used to measure the deterministic pressure error.
Our numerical tests given in Section \ref{sec-5} indicate that the
stochastic pressure error may only converge with a much slower rate in such a stronger norm.
\end{remark}

\section{Fully discrete mixed finite element discretization}\label{sec-4}
In this section we discretize the Euler-Maruyama time discretization scheme \eqref{eq3.1}
in space using the mixed finite element method. We choose the prototypical
Taylor-Hood mixed finite element method as an example and give a detailed error
analysis for the resulting fully discrete mixed finite element method.

\subsection{Preliminaries}\label{sec-4.1}
We first introduce some discrete space notation. Let $\cT_h$ be a quasi-uniform
triangulation of the polygonal or polyhedral bounded domain $D\subset \R^d$ into triangles
when $d$=2 and tetrahedra when $d=3$, respectively. We define the following two
finite element spaces:
\begin{align*}
V_h &=\bigl\{v_h\in C(\overline{D});\, v_h|_K\in P_{l+1}(K)\,\,\forall K\in  \cT_h\bigr\},\\
W_h &=\bigl\{q_h\in C(\overline{D});\, q_h|_K\in P_l(K)\,\,\forall K\in \cT_h\bigr\},
\end{align*}
where $P_l(K)$ ($l=1,2,3$) denotes the set of polynomials of degree less than or equal to $l$
over the element $K\in \cT_h$.
The Taylor-Hood mixed finite element space pair is
defined by (cf. \cite{Girault_Raviart86,Brezzi_Fortin91})
\[
 \cV_h:=[V_h\cap H^1_0(D)]^d, \qquad \cW_h:=W_h{\tiny } \cap L^2_0(D).
\]
We also set
\[
\mV_h:=L^2(\Omega, \cV_h),\qquad \mW_h:=L^2(\Omega, \cW_h).
\]
Moreover, we define the weakly divergence-free subspace $\cV_{0h}$ of $\cV_{h}$ as
\begin{equation}\label{weakly_divergence-free}
\cV_{0h}=\Bigl\{v_h\in\cV_h;\, b(v_h,q_h)=0 \, \forall q_h \in \cW_h \Bigr\}.
\end{equation}

It is well-known \cite{Brezzi_Fortin91} that the Taylor-Hood mixed finite element space pair
is stable in the sense that they satisfy the following discrete inf-sup condition: there
exists an $h$-independent positive constant $\gamma$ such that
\begin{equation}\label{inf-sup-TH}
\sup_{v_h\in \cW_h} \frac{ b(v_h,q_h)}{\|v_h\|_{\cV}} \geq \gamma \|q_h\|_{\cW}
\qquad \forall q_h\in \cW_h.
\end{equation}
Its stochastic counterpart is given by the following lemma.

\begin{lemma}\label{lem4.1}
There exists a positive constant $\hat{\gamma}$ independent of $h$ such that
\begin{align}\label{eq4.3}
\sup_{v_h\in \mV_h} \frac{E\bigl[b(v_h,q_h)\bigr]}{\|v_h\|_{\mV}} \geq \hat{\gamma}\|q_h\|_{\mW}
\qquad\forall q_h\in \mW_h.
\end{align}
\end{lemma}

\begin{proof}
We first like to comment that since $b(\cdot,\cdot)$ is a bilinear form $\mP$-a.s.,  the proof of \eqref{eq4.3}
essentially follows from the same lines of the proof for the deterministic inf-sup condition
and taking the expectation on each inequality appeared in that proof.
However, below we present a proof for \eqref{eq4.3} for the sake of completeness.
	
For any $q_h\in \mW_h\subset \mW$, from the proof of Lemma \ref{lem2.1} we know that
there exists $z\in \mV$ such that $\div z=q_h$ in $L^2(\Omega\times D)$ and there exists
a constant $c_*>0$ such that
\begin{align}\label{eq4.4}
\|z\|_{\mV}\leq c_* \|q_h\|_{\mW}.
\end{align}
Since the Taylor-Hood mixed finite element space pair is stable, it follows from
Fortin's equivalence lemma (cf. \cite{Brezzi_Fortin91}) that there exists a linear operator
$\Pi_h: \cV\to \cV_h$ such that for any $v\in \cV$ there exists a constant $c^*>0$ such that
\begin{align*}
b\bigl(\Pi_h v,\phi_h\bigr) &=b(v,\phi_h) \qquad\forall \phi_h\in \cW_h,\\
\|\Pi_h v\|_{\cV} &\leq c^*\|v\|_{\cV}.
\end{align*}
Extending trivially the domain of $\Pi_h$ to $\mV$ (with the range $\mV_h$) by
\begin{align}\label{eq4.5}
E\bigl[b\bigl(\Pi_h v,\phi_h\bigr)\bigr]=E\bigl[b(v,\phi_h)\bigr] \qquad\forall \phi_h\in \mW_h
\end{align}
for any $v\in \mV$, then we have
\begin{equation}\label{eq4.5a}
\|\Pi_h v\|_{\mV}\leq c^*\|v\|_{\mV}.
\end{equation}

Now, let $z_h:=\Pi_h z\in\mV_h$, from \eqref{eq4.5a} and \eqref{eq4.4} we get
\begin{align}\label{eq4.6}
\|z_h\|_{\mV}\leq c^* \|z\|_{\mV} \leq c^*c_*\|q_h\|_{\mW}.
\end{align}
It follows from \eqref{eq4.5} and \eqref{eq4.6} that
\begin{align*}
\sup_{v_h\in \mV_h}\frac{E\bigl[b(v_h,q_h)\bigr]}{\|v_h\|_{\mV}}
\geq\frac{E\bigl[b(z_h,q_h)\bigr]}{\|z_h\|_{\mV}}
=\frac{E\bigl[b(z,q_h)\bigr]}{\|z_h\|_{\mV}}
=\frac{E[\|q_h\|_{L^2}^2]}{\|z_h\|_{\mV}}
\geq \frac{1}{c^*c_*}\|q_h\|_{\mW},
\end{align*}
Thus, \eqref{eq4.3} holds with $\hat{\gamma}=\frac{1}{c^*c_*}$. The proof is complete.
\end{proof}

Finally, let $\rho_h: L^2(D)\to \cW_h$ denote the $L^2$-projection operator,  we cite the following
well-known approximation properties of $\rho_h$ and a Fortin operator $\Pi_h$ for the Taylor-Hood
element (cf. \cite{Ern_Guermond04,Girault_Raviart86,Falk08}):
\begin{align}\label{eq4.1}
&\|v-\Pi_h v\|_{L^2} +h\|\nabla(v-\Pi_h v)\|_{L^2} \leq C_{10} h^{r}\,\|v\|_{H^r} \,\,\forall v\in [H^r(D)]^d,\\
\label{eq4.2}
&\|\varphi-\rho_h\varphi\|_{L^2}+ h\|\nabla(\varphi-\rho_h\varphi)\|_{L^2}
\leq C_{10} h^s\|\varphi\|_{H^s} \,\, \forall \varphi\in H^s(D)
\end{align}
for $r=1,2, 3; s=1,2$. Where $C_{10}$ is a positive constant independent of $h$.

\subsection{Formulation of fully discrete mixed finite element method and its error analysis}\label{sec-4.2}
Our fully discrete finite element method for \eqref{eq2.8} is defined simply by adding a sub-index $h$
to all the functions and spaces appearing in the semi-discrete scheme \eqref{eq3.1}. Specifically,
we seek $\{\cF_{t_n};1\leq n\leq N\}$-adapted processes $\{u^n_h;1\leq n\leq N\} \in {\mathbb V}_h$
and $\{p^n_h;1\leq n\leq N\} \in {\mathbb W}_h$ such that for any $(v_h,q_h)\in \cV_h\times \cW_h $
\begin{subequations}\label{eq4.8}
\begin{align}\label{eq4.8a}
&\bigl(u_h^{n+1},v_h\bigr) +k\,\nu a(u_h^{n+1},v_h) + k\,b(v_h,p^{n+1}_h)  \\ \nonumber
&\hskip 0.6in  =\bigl(u_h^{n},v_h\bigr) +\int_{t_n}^{t_{n+1}}\bigl(f(s),v_h\bigr)\,ds
+\bigl(B(t_n,u^n_h)\Delta W_{n+1},v_h \bigr) \quad a.s.\\
&\bigl(\div u_h^{n+1},q_h\bigr) =0 \quad a.s.   \label{eq4.8b}
\end{align}
\end{subequations}
where $u^0_h=\cP_h u_0\in \cV_{0h}$, the $L^2$ orthogonal projection of $u_0$ into $\cV_{0h}$.

We first state the following stability estimates for $\{u_h^{n}\}$ and $\{p_h^n\}$ but omit their
proofs because they are similar to those of their semi-discrete counterparts given in Lemma \ref{lem3.1}.

\begin{lemma}
Let $\{ (u^n_h, p_h^n); 1\leq n\leq N \}$ be a solution to scheme \eqref{eq4.8}, then there hold
\begin{align}\label{eq4.10}
&\max_{1\leq n\leq N} E\bigl[\|u^n_h\|^2_{L^2}\bigr] +E\Bigl[\sum^N_{n=1}\|u^n_h-u^{n-1}_h\|^2_{L^2}\Bigr]
+\nu E\Bigl[k\sum^N_{n=1}\|\nabla u^n_h\|^2_{L^2}\Bigr] \\ \nonumber
&\hskip 1in
\leq C_{11}\bigl\{E\bigl[\|u^0_h\|^2_{L^2}\bigr] +E\bigl[ \|f\|^2_{L^2(0,T; H^{-1})}\bigr] \Bigr\},\\
\label{eq4.11}
&E\Bigl[ \Bigl\| k\sum_{n=1}^N p_h^n \Bigr\|_{L^2}^2\Bigr]
\leq C_{11}\Bigl\{E\bigl[\|u_h^0\|^2_{L^2}\bigr]  +E\bigl[ \|f\|^2_{L^2(0,T; H^{-1})} \bigr] \Bigr\},
\end{align}
where $C_{11}= C(C_T,D,T,\hat{\gamma})>0$ is a constant that does not depend on $k$ and $h$.
\end{lemma}

\begin{remark}
(a)	We note that a similar estimate to \eqref{eq4.10} was proved in \cite{BCP12} but
	\eqref{eq4.11} seems new.  We also emphasize that the above stability estimates will not be used
	in our error analysis, instead, those given in Lemma \ref{lem3.1} will be crucially used.
	
(b) Since \eqref{eq4.8} is equivalent to a linear system, the above stability estimates immediately
infer the well-posedness of scheme \eqref{eq4.8}.
\end{remark}

We now are ready to state the first main theorem of this section.

\begin{theorem}\label{thm4.1}
Set $u^0=u_0$ and let  $\{ (u^n, p^n); 1\leq n\leq N\}$ and
$\{ (u^n_h, p_h^n); 1\leq n\leq N \}$ be the solutions of \eqref{eq3.1} and \eqref{eq4.8},
respectively. Then there holds for $l\geq 1$
\begin{align}\label{eq4.13}
\max_{1\leq n\leq N}    E\bigl[\|u^n-u^n_h\|^2_{L^2}\bigr]
+ E\Bigl[k\sum_{n=1}^{N}\|\nabla (u^n-u^n_h)\|^2_{L^2}\Bigr]   \leq C_{12} k^{-1} h^{2},
\end{align} 
and for $l\geq 2$ 
\begin{align} \label{eq4.13a}
\max_{1\leq n\leq N}    E\bigl[\|\nabla(u^n-u^n_h)\|^2_{L^2}\bigr]
+ E\Bigl[k\sum_{n=1}^{N}\|A(u^n-u^n_h)\|^2_{L^2}\Bigr]  \leq C_{13}   k^{-1} h^{2},
\end{align} 
and for $l=3$
\begin{align} \label{eq4.13b}
\max_{1\leq n\leq N}   E\bigl[\|A(u^n-u^n_h)\|^2_{L^2}\bigr]
+ E\Bigl[k\sum_{n=1}^{N}\|A^{\frac{3}{2}}(u^n-u^n_h)\|^2_{L^2}\Bigr] \leq C_{14}   k^{-1} h^{2}, 
\end{align}
where $C_{12}=C(C_T,C_5,C_{10},\nu,T, M(u_0,f))$, $C_{13}=C(C_T,C_5,C_{10},\nu,T, \widetilde{M}(u_0,f))$ and $C_{14}=C(C_T,C_5,C_{10},\nu,T, \widehat{M}(u_0,f))$ are positive constants independent of $h$, and
\begin{align}\label{eq4.13c}
M(u_0,f) &:=E\bigl[\|u_0\|_{L^2}^2\bigr] +\nu k E\big[\| u_0\|^2_{H^1}\bigr] +E\bigl[\|f\|_{L^2(0,T;L^2)}^2\bigr],\\\label{eq4.13d}
\widetilde{M}(u_0,f) &:=E\bigl[\|u_0\|_{H^1}^2\bigr] +\nu k E\big[\| u_0\|^2_{H^2}\bigr] +E\bigl[\|f\|_{L^2(0,T;L^2)}^2\bigr],\\\label{eq4.13e}
\widehat{M}(u_0,f) &:=E\bigl[\|u_0\|_{H^2}^2\bigr] +\nu k E\big[\| u_0\|^2_{H^3}\bigr] +E\bigl[\|f\|_{L^2(0,T;L^2)}^2\bigr].
\end{align}
\end{theorem}

\begin{proof}
For every $n\geq1$, let $e^n=u^n-u^n_h\in \cV$ and $\xi^n=p^n-p^n_h \in \cW$, it is easy to
check that $(e^n,\xi^n)$ satisfies the following error equations:
\begin{subequations}\label{eq4.14}
\begin{align}\label{eq4.14a}
\bigl(e^{n+1}-e^n,v_h\bigr) &+k\,\nu a\bigl(e^{n+1},v_h\bigr) +k\,b\bigl(v_h,\xi^{n+1}\bigr) \\ \nonumber
&=\bigl([B(t_n,u^n)-B(t_n,u^n_h)]\Delta W_{n+1},v_h \bigr) \quad\forall v_h\in \cV_h,\\
b\bigl(e^{n+1},q_h\bigr)&=0 \quad\forall q_h\in \cW_h. \label{eq4.14b}
\end{align}
\end{subequations}

Introduce the following error decompositions:
\begin{align*}
&e^n=\theta^n + \veps^n,\quad \theta^n:=u^n-\Pi_h u^n,\quad \veps^n:= \Pi_h u^n - u_h^n,\\
&\xi^n =\tau^n + \eta^n, \quad \tau^n:=p^n-\rho_h p^n, \quad \eta^n:=\rho_h p^n-p_h^n.
\end{align*}
Setting $v_h=\veps^{n+1}$ and $q_h=\eta^{n+1}$ in \eqref{eq4.14}, taking expectation and using
the definition of $\Pi_h u^n$ and $\rho_h p^n$ we get
\begin{align}\label{eq4.15}
&\frac{1}{2}\Bigl( E\bigl[\|\veps^{n+1}\|^2_{L^2}\bigr]-E\bigl[\|\veps^n\|^2_{L^2}\bigr]
+E\bigl[\|\veps^{n+1}-\veps^n\|^2_{L^2}\bigr]\Bigr)
+\nu k\bigl[\|\nabla \varepsilon^{n+1}\|^2_{L^2} \bigr]\\ \nonumber
&\qquad
=E\Bigl[\bigl([B(t_n,u^n)-B(t_n,u^n_h)]\Delta W_{n+1},\veps^{n+1}\bigr)\Bigr]
-E\bigl[\bigl(\theta^{n+1}-\theta^n, \veps^{n+1} \bigr)\bigr] \\  \nonumber
&\qquad\qquad
-k\,E\bigl[ b\bigl(\veps^{n+1},\tau^{n+1}\bigr)\bigr] -k\nu E\bigl[ a\bigl( \theta^{n+1}, \veps^{n+1} \bigr) \bigr].
\end{align}

We now bound the terms on the right-hand side as follows. First, by \eqref{eq4.2} we get
\begin{align} \label{eq4.16}
k\,E\bigl[b\bigl(\veps^{n+1},\tau^{n+1}\bigr)\bigr]
&\leq kE\bigl[ \|\nabla \veps^{n+1}\|_{L^2} \|\tau^{n+1}\|_{L^2}\bigr]  \\
&\leq \frac{\nu k}{4} E\bigl[\|\nabla \veps^{n+1}\|_{L^2}^2\bigr]
+\frac{k}{\nu} E\bigl[\|\tau^{n+1}\|_{L^2}^2\bigr] \nonumber\\
&\leq \frac{\nu k}{4} E\bigl[\|\nabla \veps^{n+1}\|_{L^2}^2 \bigr]
+\frac{C_6^2 k}{\nu} E\bigl[  \|p^{n+1}\|_{H^{1}}^2\bigr] h^{2}.
\nonumber \\
E\bigl[\bigl(\theta^{n+1}-\theta^n, \veps^{n+1}\pm \veps^n \bigr)\bigr]
&\leq E\bigl[ \|\theta^{n+1}-\theta^n\|_{L^2}\,\bigl( \|\veps^{n+1}-\veps^n \|_{L^2} + \|\veps^n\| \bigr) \bigr] \label{eq4.17} \\ \nonumber
&\hskip -1.2in
\leq\frac18 E\bigl[\|\veps^{n+1}-\veps^n\|_{L^2}^2] + (2+k^{-1})E\bigl[\|\theta^{n+1}-\theta^n\|_{L^2}^2 \bigr]
+k E\bigl[\|\veps^n\|^2_{L^2}  \bigr] \nonumber \\
&\hskip -1.2in
\leq\frac18 E\bigl[\|\veps^{n+1}-\veps^n\|_{L^2}^2]
+2C_6^2 E\bigl[\|u^{n+1}-u^n\|_{H^2}^2\bigr] \,k^{-1}h^{4} +k E\bigl[\|\veps^n\|^2_{L^2} \bigr]. \nonumber \\
k\nu E\bigl[ a\bigl( \theta^{n+1}, \veps^{n+1} \bigr) \bigr]
&\leq   \frac{\nu k}{4} E\bigl[\|\nabla \veps^{n+1}\|_{L^2}^2 \bigr]
+\nu k E\bigl[\|\nabla \theta^{n+1}\|_{L^2}^2\bigr]   \label{eq4.17a} \\
&\leq  \frac{\nu k}{4} E\bigl[\|\nabla \veps^{n+1}\|_{L^2}^2 \bigr]
+C_6^2 \nu kE\bigl[  \|u^{n+1}\|_{H^{2}}^2\bigr] h^{2}. \nonumber
\end{align}
\begin{align}\label{eq4.18}
&E\bigl[\bigl([B(t_n,u^n)-B(t_n,u_h^n)]\Delta W_{n+1},\veps^{n+1}\bigr)\bigr] \\ \nonumber
&\quad =E\bigl[\bigl([B(t_n,u^n)-B(t_n,u_h^n)]\Delta W_{n+1},\veps^{n+1}-\veps^n \bigr)\bigr] \\ \nonumber
&\quad
\leq E\bigl[\|[B(t_n,u^n)-B(t_n,u_h^n)]\Delta W_{n+1}\|_{L^2} \|\veps^{n+1}-\veps^n\|_{L^2}\bigr] \\ \nonumber
&\quad
\leq \frac18 E\bigl[\|\veps^{n+1}-\veps^n\|_{L^2}^2]+ 2E\bigl[\|[B(t_n,u^n)-B(t_n,u_h^n)]\Delta W_{n+1}\|_{L^2}^2\bigr] \\
\nonumber
&\quad
\leq \frac18 E\bigl[\|\veps^{n+1}-\veps^n\|_{L^2}^2]+ 2k E\bigl[\|B(t_n,u^n)-B(t_n,u_h^n)\|_{\cK}^2\bigr]\\ \nonumber
&\quad
\leq \frac18 E\bigl[\|\veps^{n+1}-\veps^n\|_{L^2}^2]+ 2C_T^2k E\bigl[\|u^n-u_h^n\|_{L^2}^2\bigr]\\ \nonumber
&\quad
\leq  \frac18 E\bigl[\|\veps^{n+1}-\veps^n\|_{L^2}^2]+4C_T^2k E\bigl[\|\veps^n\|_{L^2}^2 +\|\theta^n\|_{L^2}^2\bigr]\\ \nonumber
&\quad
\leq  \frac18 E\bigl[\|\veps^{n+1}-\veps^n\|_{L^2}^2]+4C_T^2k E\bigl[\|\veps^n\|_{L^2}^2\bigr]
+4C_TC_6^2k E\bigl[ \|u^n\|^2_{H^2} \bigr] h^4. 
\end{align}
Here we have used \eqref{eq4.1} and \eqref{eq4.2} to obtain the last inequalities
in \eqref{eq4.16}--\eqref{eq4.18}.

Applying the summation operator $\sum_{n=0}^{N-1}$ on both sides of \eqref{eq4.15} and using
estimates \eqref{eq4.16}--\eqref{eq4.18} and Lemma \ref{lem3.1} we obtain
\begin{align}\label{eq4.19}
&E\bigl[\|\veps^{N}\|^2_{L^2}\bigr]+\sum^{N-1}_{n=0} E\bigl[\|\veps^{n+1}-\veps^n\|^2_{L^2}\bigr]
+\nu k\sum^{N-1}_{n=0} E\bigl[\|\nabla \veps^{n+1}\|^2_{L^2}\bigr]\\  \nonumber
&\qquad \leq 16C_T^2 k \sum^{N-1}_{n=0} E\bigl[\|\veps^n\|^2_{L^2}\bigl]
+ 8C_6^2 \sum_{n=0}^{N-1} E\bigl[\|u^{n+1}-u^n\|_{H^2}^2\bigr] h^{4} k^{-1} \\ \nonumber
&\qquad \quad
+4C_6^2\Bigl\{ (\nu+ 4C_T^2h^2) k\sum_{n=0}^{N} E\bigl[\|u^n\|^2_{H^2}\bigr] 
+\nu^{-1} k\sum_{n=0}^{N-1} E\bigl[\|p^{n+1}\|_{H^{1}}^2\bigr]\Bigr\} h^{2} \\ \nonumber
&\qquad   \leq 16C_T^2 k \sum^{N-1}_{n=0} E\bigl[\|\veps^n\|^2_{L^2}\bigl]
+ \widetilde{C} M(u_0,f) h^{2},
\end{align}
where $\widetilde{C}_1:= 16C_5C_{10}^2\bigl[(4C_T^2+ \nu +1)h^2 k^{-1} + \nu^{-1} k^{-1} \bigr]$ and
$M(u_0,f)$ is defined by \eqref{eq4.13c}.

It follows from \eqref{eq4.19} and Gronwall's inequality that
\begin{align}\label{eq4.21}
\max_{1\leq n\leq N} E\bigl[\|\veps^n\|^2_{L^2}\bigr]
+k\nu\sum^N_{n=1}E\bigl[\|\nabla\veps^n\|^2_{L^2}\bigr]
\leq \exp(16C_T^2 T)\, \widetilde{C} M(u_0,f)  \, h^{2}.
\end{align}
Then \eqref{eq4.13} follows from an application of the triangle inequality
on $e^n=\theta^n + \veps^n$.

To show \eqref{eq4.13a}, setting $v_h=A\veps^{n+1} \in \cV_{0h}$ and $q_h=0$ in \eqref{eq4.14},
taking expectation and using the definition of $\Pi_h u^n$ we get
\begin{align}\label{eq4.210}
&\frac{1}{2}\Bigl( E\bigl[\|\nabla\veps^{n+1}\|^2_{L^2}\bigr]-E\bigl[\|\nabla\veps^n\|^2_{L^2}\bigr]
+E\bigl[\|\nabla(\veps^{n+1}-\veps^n)\|^2_{L^2}\bigr]\Bigr) \\ \nonumber
&\qquad \qquad +\nu kE\bigl[\|A \varepsilon^{n+1}\|^2_{L^2} \bigr]\\ \nonumber
&=E\Bigl[\bigl(\nabla[B(t_n,u^n)-B(t_n,u^n_h)]\Delta W_{n+1},\nabla\veps^{n+1}\bigr)\Bigr] \\ \nonumber
&\qquad \qquad -E\bigl[\bigl(\nabla(\theta^{n+1}-\theta^n), \nabla\veps^{n+1} \bigr)\bigr]
-k\nu E\bigl[ \bigl( A\theta^{n+1}, A\veps^{n+1} \bigr) \bigr].
\end{align}
We estimate the terms on the right-hand side as follows. Using \eqref{eq4.1}--\eqref{eq4.2} we have
\begin{align} \label{eq4.211}
& E\bigl[\bigl(\nabla(\theta^{n+1}-\theta^n), \nabla(\veps^{n+1}\pm \veps^n) \bigr)\bigr] \\ \nonumber
& \leq E\bigl[ \|\nabla(\theta^{n+1}-\theta^n)\|_{L^2}\,\bigl( \|\nabla(\veps^{n+1}-\veps^n) \|_{L^2} + \|\nabla\veps^n\| \bigr) \bigr]  \\ \nonumber
& \leq\frac18 E\bigl[\|\nabla(\veps^{n+1}-\veps^n)\|_{L^2}^2] + (2+k^{-1})E\bigl[\|\nabla(\theta^{n+1}-\theta^n)\|_{L^2}^2 \bigr]
+k E\bigl[\|\nabla\veps^n\|^2_{L^2}  \bigr] \nonumber \\
& \leq\frac18 E\bigl[\|\nabla(\veps^{n+1}-\veps^n)\|_{L^2}^2]
+2C_{10}^2 E\bigl[\|\nabla(u^{n+1}-u^n)\|_{H^2}^2\bigr] \,k^{-1}h^{4} +k E\bigl[\|\nabla\veps^n\|^2_{L^2} \bigr]. \nonumber
\end{align}
\begin{align}
k\nu E\bigl[\bigl( A\theta^{n+1}, A\veps^{n+1} \bigr) \bigr]
&\leq   \frac{\nu k}{4} E\bigl[\|A \veps^{n+1}\|_{L^2}^2 \bigr]
+\nu k E\bigl[\|A \theta^{n+1}\|_{L^2}^2\bigr]   \label{eq4.211a} \\
&\leq  \frac{\nu k}{4} E\bigl[\|A \veps^{n+1}\|_{L^2}^2 \bigr]
+C_{10}^2 \nu kE\bigl[  \|u^{n+1}\|_{H^{3}}^2\bigr] h^{2}. \nonumber
\end{align}
\begin{align}\label{eq4.212}
&E\bigl[\bigl(\nabla[B(t_n,u^n)-B(t_n,u_h^n)]\Delta W_{n+1},\nabla\veps^{n+1}\bigr)\bigr] \\ \nonumber
&\quad =E\bigl[\bigl(\nabla[B(t_n,u^n)-B(t_n,u_h^n)]\Delta W_{n+1},\nabla(\veps^{n+1}-\veps^n) \bigr)\bigr] \\ \nonumber
&\quad
\leq E\bigl[\|\nabla[B(t_n,u^n)-B(t_n,u_h^n)]\Delta W_{n+1}\|_{L^2} \|\nabla(\veps^{n+1}-\veps^n)\|_{L^2}\bigr] \\ \nonumber
&\quad
\leq \frac18 E\bigl[\|\nabla(\veps^{n+1}-\veps^n)\|_{L^2}^2]+ 2E\bigl[\|\nabla[B(t_n,u^n)-B(t_n,u_h^n)]\Delta W_{n+1}\|_{L^2}^2\bigr] \\
\nonumber
&\quad
\leq \frac18 E\bigl[\|\nabla(\veps^{n+1}-\veps^n)\|_{L^2}^2]+ 2k E\bigl[\|\nabla(B(t_n,u^n)-B(t_n,u_h^n))\|_{\cK}^2\bigr]\\ \nonumber
&\quad
\leq  \frac18 E\bigl[\|\nabla(\veps^{n+1}-\veps^n)\|_{L^2}^2]+4C_T^2k E\bigl[\|\nabla\veps^n\|_{L^2}^2\bigr]
+4C_TC_{10}^2k E\bigl[ \|u^n\|^2_{H^3} \bigr] h^4. 
\end{align}
Here we have used \eqref{eq4.1} to obtain the last inequalities
in \eqref{eq4.211}--\eqref{eq4.212}.
Applying the summation operator $\sum_{n=0}^{N-1}$ on both sides of \eqref{eq4.210} and using
estimates \eqref{eq4.211}--\eqref{eq4.212} and Lemma \ref{lem3.1} we obtain
\begin{align}\label{eq4.213}
&E\bigl[\|\nabla\veps^{N}\|^2_{L^2}\bigr]+\sum^{N-1}_{n=0} E\bigl[\|\nabla(\veps^{n+1}-\veps^n)\|^2_{L^2}\bigr]
+\nu k\sum^{N-1}_{n=0} E\bigl[\|A \veps^{n+1}\|^2_{L^2}\bigr]\\  \nonumber
&\qquad \leq 16C_T^2 k \sum^{N-1}_{n=0} E\bigl[\|\nabla\veps^n\|^2_{L^2}\bigl]
+ 8C_{10}^2 \sum_{n=0}^{N-1} E\bigl[\|\nabla(u^{n+1}-u^n)\|_{H^2}^2\bigr] h^{4} k^{-1} \\ \nonumber
&\qquad \quad
+4C_{10}^2\Bigl\{ (\nu+ 4C_T^2h^2) k\sum_{n=0}^{N} E\bigl[\|u^n\|^2_{H^3}\bigr] 
 \\ \nonumber
&\qquad   \leq 16C_T^2 k \sum^{N-1}_{n=0} E\bigl[\|\nabla\veps^n\|^2_{L^2}\bigl]
+ \widetilde{C}_2 M(u_0,f) h^{2},
\end{align}
where  $\widetilde{C}_2:= 16C_5C_{10}^2\bigl[(4C_T^2+ \nu +1)h^2 k^{-1}\bigr]$ and
$\widehat{M}(u_0,f)$ is defined by \eqref{eq4.13d}.

It follows from \eqref{eq4.213} and Gronwall's inequality that
\begin{align}\label{eq4.214}
\max_{1\leq n\leq N} E\bigl[\|\nabla\veps^n\|^2_{L^2}\bigr]
+k\nu\sum^N_{n=1}E\bigl[\|A\veps^n\|^2_{L^2}\bigr]
\leq \exp(16C_T^2 T)\, \widetilde{C} M(u_0,f)  \, h^{2}.
\end{align}
Thus, \eqref{eq4.13a} follows from an application of the triangle inequality
on $e^n=\theta^n + \veps^n$.

To prove \eqref{eq4.13b},
setting $v_h=A^{2}\veps^{n+1} \in \cV_{0h}$ and $q_h=0$ in \eqref{eq4.14}, taking expectation and using
the definition of $\Pi_h u^n$ we get
\begin{align}\label{eq4.215}
&\frac{1}{2}\Bigl( E\bigl[\|A\veps^{n+1}\|^2_{L^2}\bigr]-E\bigl[\|A\veps^n\|^2_{L^2}\bigr]
+E\bigl[\|A(\veps^{n+1}-\veps^n)\|^2_{L^2}\bigr]\Bigr) \\ \nonumber
&\hskip 1.1in  +\nu kE\bigl[\|A^{\frac{3}{2}} \varepsilon^{n+1}\|^2_{L^2} \bigr]\\ \nonumber
&\qquad
=E\Bigl[\bigl(A[B(t_n,u^n)-B(t_n,u^n_h)]\Delta W_{n+1},A\veps^{n+1}\bigr)\Bigr] \\ \nonumber
&\qquad\qquad  -E\bigl[\bigl(A(\theta^{n+1}-\theta^n), A\veps^{n+1} \bigr)\bigr]
-k\nu E\bigl[ \bigl( A^{\frac{3}{2}}\theta^{n+1}, A^{\frac{3}{2}}\veps^{n+1} \bigr) \bigr].
\end{align}
We now bound the terms on the right-hand side as follows. First, by \eqref{eq4.2} we get
\begin{align} \label{eq4.216}
&E\bigl[\bigl(A(\theta^{n+1}-\theta^n), A(\veps^{n+1}\pm \veps^n) \bigr)\bigr] \\ \nonumber
&\leq E\bigl[ \|A(\theta^{n+1}-\theta^n)\|_{L^2}\,\bigl( \|A(\veps^{n+1}-\veps^n) \|_{L^2} + \|A\veps^n\| \bigr) \bigr]  \\ \nonumber
&\leq\frac18 E\bigl[\|A(\veps^{n+1}-\veps^n)\|_{L^2}^2] + (2+k^{-1})E\bigl[\|A(\theta^{n+1}-\theta^n)\|_{L^2}^2 \bigr]
+k E\bigl[\|A\veps^n\|^2_{L^2}  \bigr] \nonumber \\
&\leq\frac18 E\bigl[\|A(\veps^{n+1}-\veps^n)\|_{L^2}^2]
+2C_{10}^2 E\bigl[\|A(u^{n+1}-u^n)\|_{H^2}^2\bigr] \,k^{-1}h^{4} +k E\bigl[\|A\veps^n\|^2_{L^2} \bigr]. \nonumber
\end{align}
\begin{align}
k\nu E\bigl[\bigl( A^{\frac{3}{2}}\theta^{n+1}, A^{\frac{3}{2}}\veps^{n+1} \bigr) \bigr]
&\leq   \frac{\nu k}{4} E\bigl[\|A^{\frac{3}{2}} \veps^{n+1}\|_{L^2}^2 \bigr]
+\nu k E\bigl[\|A \theta^{n+1}\|_{L^2}^2\bigr]   \label{eq4.216a} \\
&\leq  \frac{\nu k}{4} E\bigl[\|A^{\frac{3}{2}} \veps^{n+1}\|_{L^2}^2 \bigr]
+C_{10}^2 \nu kE\bigl[  \|u^{n+1}\|_{H^{4}}^2\bigr] h^{2}. \nonumber
\end{align}
\begin{align}\label{eq4.217}
&E\bigl[\bigl(A[B(t_n,u^n)-B(t_n,u_h^n)]\Delta W_{n+1},A\veps^{n+1}\bigr)\bigr] \\ \nonumber
&\quad =E\bigl[\bigl(A[B(t_n,u^n)-B(t_n,u_h^n)]\Delta W_{n+1},A(\veps^{n+1}-\veps^n) \bigr)\bigr] \\ \nonumber
&\quad
\leq E\bigl[\|A[B(t_n,u^n)-B(t_n,u_h^n)]\Delta W_{n+1}\|_{L^2} \|A(\veps^{n+1}-\veps^n)\|_{L^2}\bigr] \\ \nonumber
&\quad
\leq \frac18 E\bigl[\|A(\veps^{n+1}-\veps^n)\|_{L^2}^2]+ 2E\bigl[\|A[B(t_n,u^n)-B(t_n,u_h^n)]\Delta W_{n+1}\|_{L^2}^2\bigr] \\
\nonumber
&\quad
\leq \frac18 E\bigl[\|A(\veps^{n+1}-\veps^n)\|_{L^2}^2]+ 2k E\bigl[\|A(B(t_n,u^n)-B(t_n,u_h^n))\|_{\cK}^2\bigr]\\ \nonumber
&\quad
\leq  \frac18 E\bigl[\|A(\veps^{n+1}-\veps^n)\|_{L^2}^2]+4C_T^2k E\bigl[\|A\veps^n\|_{L^2}^2\bigr]
+4C_TC_{10}^2k E\bigl[ \|u^n\|^2_{H^4} \bigr] h^4. 
\end{align}
Here we have used \eqref{eq4.1} to obtain the last inequalities
in \eqref{eq4.216}--\eqref{eq4.217}.
Applying the summation operator $\sum_{n=0}^{N-1}$ on both sides of \eqref{eq4.215} and using
estimates \eqref{eq4.216}--\eqref{eq4.217} and Lemma \ref{lem3.1} we obtain
\begin{align}\label{eq4.218}
&E\bigl[\|A\veps^{N}\|^2_{L^2}\bigr]+\sum^{N-1}_{n=0} E\bigl[\|A(\veps^{n+1}-\veps^n)\|^2_{L^2}\bigr]
+\nu k\sum^{N-1}_{n=0} E\bigl[\|A^{\frac{3}{2}} \veps^{n+1}\|^2_{L^2}\bigr]\\  \nonumber
&\qquad \leq 16C_T^2 k \sum^{N-1}_{n=0} E\bigl[\|A\veps^n\|^2_{L^2}\bigl]
+ 8C_{10}^2 \sum_{n=0}^{N-1} E\bigl[\|A(u^{n+1}-u^n)\|_{H^2}^2\bigr] h^{4} k^{-1} \\ \nonumber
&\qquad \quad
+4C_{10}^2\Bigl\{ (\nu+ 4C_T^2h^2) k\sum_{n=0}^{N} E\bigl[\|u^n\|^2_{H^4}\bigr] 
 \\ \nonumber
&\qquad   \leq 16C_T^2 k \sum^{N-1}_{n=0} E\bigl[\|A\veps^n\|^2_{L^2}\bigl]
+ \widetilde{C} M(u_0,f) h^{2},
\end{align}
where  $\widetilde{C}_3:= 16C_5C_{10}^2\bigl[(4C_T^2+ \nu +1)h^2 k^{-1}\bigr]$ and
$\widehat{M}(u_0,f)$ is defined by \eqref{eq4.13b}.

It follows from \eqref{eq4.218} and Gronwall's inequality that
\begin{align}\label{eq4.219}
\max_{1\leq n\leq N} E\bigl[\|A\veps^n\|^2_{L^2}\bigr]
+k\nu\sum^N_{n=1}E\bigl[\|A^{\frac{3}{2}}\veps^n\|^2_{L^2}\bigr]
\leq \exp(16C_T^2 T)\, \widetilde{C}_3 M(u_0,f)  \, h^{2}.
\end{align}

Finally, \eqref{eq4.13e} follows from an application of the triangle inequality
on $e^n=\theta^n + \veps^n$. The proof is complete.
\end{proof}

\begin{remark}
(a) We note that the conclusion of the theorem still holds if $u_h^0=\cP_h u_0$ is replaced
by $u_h^0=\cQ_h u_0$, the $L^2$-projection of $u_0$ into $\cV_h$. We emphasize that
the $k^{-\frac12}$ factor in the error bound is a reflection of the low regularity of the pressure $p$
and the simultaneous approximation property of the mixed finite element method.

(b) We also note that the error estimate for the velocity approximations of divergence-free finite element
methods do not have the ``bad" factor $k^{-\frac12}$ (cf. \cite{CP12}) at the expense of using
divergence-free finite element spaces and not approximating the pressure.
\end{remark}

The second main result of this section is the following error estimate for the pressure
approximation.

\begin{theorem}\label{thm4.2}
Under the assumptions of Theorem \ref{thm4.1} there holds  
\begin{align}\label{eq4.22}
E\Bigl[\Bigl\|k\sum^N_{n=1} (p^n- p^n_h) \Bigr\|_{L^2}\Bigr]\leq  C_{15} k^{-\frac12} h
\qquad \mbox{for $1\leq l\leq 3$},
\end{align}
where $C_{15}=\hat{\gamma}^{-1}C_{12} \bigl((C_0+\nu^{-1}) +2C_TC_0 T^{\frac12}\bigr)>0$
is a constant independent of $h$.
\end{theorem}

\begin{proof}
The proof follows the same lines as the proof of Theorem \ref{thm3.3}.
First, summing \eqref{eq4.8a} (after lowering the index by one) over $1\leq n\leq m (\leq N)$
and subtracting  the resulting equation from \eqref{eq3.21a} we get
\begin{align*}
&\bigl(e^m,v_h\bigr) +\nu k\sum_{n=1}^m a\bigl(e^n, v_h\bigr)+k\sum_{n=1}^m b\bigl(v_h,\xi^n\bigr)\\
&\quad
=(e^0,v_h)+ \sum_{n=1}^m \bigl([B(t_{n-1},u^{n-1})-B(t_{n-1},u^{n-1}_h)] \Delta W_n,v_h\bigr)
\quad\forall v_h\in \cV_h,\,\, a.s
\end{align*}
Here $e^n$ and $\xi^n$ are the same as in the proof of the previous theorem.  Then we have
\begin{align}\label{eq4.23}
&E\Bigl[b\Bigl(v_h,k\sum^m_{n=1} \xi^n \Bigr)\Bigr]
=E\bigl[\bigl(e^0-e^n,v_h\bigr)]-\nu k\sum^m_{n=1} E\bigl[a\bigl(e^n,v_h\bigr)\bigr]\\ \nonumber
&\qquad\qquad
+\sum^m_{n=1} E\bigl[\bigl([B(t_{n-1},u^{n-1})-B(t_{n-1},u^{n-1}_h)]\Delta W_n,v_h\bigr) \bigr]\\ \nonumber
&\, \leq (C_0+\frac{1}{\nu})\Bigl( E\bigl[\|e^0\|_{L^2}^2 + \|e^n\|_{L^2}^2\bigr]
+ \nu k \sum_{n=1}^m E\bigl[\|\nabla e^n\|_{L^2}^2\bigr] \Bigr)^{\frac12}
\Bigl(E\bigl[\|\nabla v_h\|_{L^2}^2\bigr]\Bigr)^{\frac12}\\ \nonumber
&\qquad\qquad
+\sum^m_{n=1} E\bigl[\bigl([B(t_{n-1},u^{n-1})-B(t_{n-1},u^{n-1}_h)]\Delta W_n,v_h\bigr) \bigr] \\ \nonumber
&\, \leq (C_0+\frac{1}{\nu}) C_8 k^{-\frac12} h \|v_h\|_{\mV}\\ \nonumber
&\qquad\qquad
+\sum^m_{n=1} E\bigl[\bigl([B(t_{n-1},u^{n-1})-B(t_{n-1},u^{n-1}_h)]\Delta W_n,v_h\bigr) \bigr].
\end{align}
Using the Poincar\'{e} inequality, the last term in \eqref{eq4.23} can be bounded as
\begin{align}\label{eq4.24}
&\sum^m_{n=1} E\bigl[\bigl([B(t_{n-1},u^{n-1})-B(t_{n-1},u^{n-1}_h)]\Delta W_n,v_h\bigr) \bigr] \\ \nonumber
&\quad
\leq E\Bigl[ \Bigl\|\sum^m_{n=1}[B(t_{n-1},u^{n-1})-B(t_{n-1},u^{n-1}_h)]\Delta W_n\Bigr\|_{L^2} \|v_h\|_{L^2}\Bigr]\\ \nonumber
&\quad
\leq 2\Bigl( \sum_{n=1}^m E\bigl[\|[B(t_{n-1},u^{n-1})-B(t_{n-1},u^{n-1}_h)]\Delta W_n\|^2_{L^2} \Bigr)^{\frac12} \Bigl( E\bigl[\|v_h\|^2_{L^2}\bigr] \Bigr)^{\frac12}\\ \nonumber
&\quad
\leq 2\Bigl( k\sum_{n=1}^m E\bigl[ \|B(t_{n-1},u^{n-1})-B(t_{n-1},u^{n-1}_h)\|^2_{\cK}\bigr] \Bigr)^{\frac12} \Bigl( E\bigl[\|v_h\|^2_{L^2}\bigr] \Bigr)^{\frac12}\\ \nonumber
&\quad
\leq 2C_TC_0\Bigl( k\sum_{n=1}^m E\bigl[ \|e^{n-1}\|^2_{L^2}\bigr] \Bigr)^{\frac12}\Bigl(E\bigl[\|\nabla v_h\|^2_{L^2}\bigr] \Bigr)^{\frac12}\\ \nonumber
&\quad
\leq 2C_TC_0 T^{\frac12} C_{12}k^{-\frac12}  h \|v_h\|_{\mV}.
\end{align}

Finally, it follows from \eqref{eq4.23}--\eqref{eq4.24} and the discrete inf-sup condition
\eqref{eq4.3} that
\begin{align*}
\hat{\gamma} \Bigl\| k\sum^m_{n=1} \xi^n \Bigr\|_{\mW}
\leq  C_{12}\bigl( (C_0+\nu^{-1}) +2C_TC_0 T^{\frac12}\bigr) k^{-\frac12} h,
\end{align*}
which gives the desired inequality \eqref{eq4.15}. The proof is complete.
\end{proof}

Theorems \ref{thm3.2}, \ref{thm3.3}, \ref{thm4.1} and Theorem \ref{thm4.2} and the triangle inequality
immediately infer the following global error estimates, which are the main results of this paper.

\begin{theorem}
Under the assumptions of Theorems \ref{thm3.2}, \ref{thm3.3}, \ref{thm4.1} and Theorem \ref{thm4.2},
there hold the following estimates:
\begin{align}\label{eq4.26}
\max_{1\leq n\leq N} \Bigl(E\bigl[\|u(t_n)-u^n_h\|^2_{L^2}\bigl]\Bigr)^{\frac12}
&+\Bigl( E\Bigl[ k\sum_{n=1}^{N}\|\nabla (u(t_n)-u^n_h)\|^2_{L^2}\Bigr] \Bigr)^{\frac12}\\ \nonumber
&\leq C_6 k^{\frac12}+ C_{12} k^{-\frac12} h, \quad \mbox{if $l\geq 1$},  \\
\label{eq4.27}
\max_{1\leq n\leq N} \Bigl(E\bigl[\|\nabla (u(t_n)-u^n_h)\|^2_{L^2})\bigl]\Bigr)^{\frac12}
&+\Bigl( E\Bigl[ k\sum_{n=1}^{N}\|A (u(t_n)-u^n_h)\|^2_{L^2}\Bigr] \Bigr)^{\frac12}\\ \nonumber
&\leq C_7 k^{\frac12}+ C_{13} k^{-\frac12} h, \quad \mbox{if $l\geq 2$}, \\
\label{eq4.28}
\max_{1\leq n\leq N} \Bigl(E\bigl[\|A(u(t_n)-u^n_h)\|^2_{L^2})\bigl]\Bigr)^{\frac12}
&+\Bigl( E\Bigl[ k\sum_{n=1}^{N}\|A^{\frac{3}{2}}(u(t_n)-u^n_h)\|^2_{L^2}\Bigr] \Bigr)^{\frac12}\\ \nonumber
&\leq C_8 k^{\frac12}+ C_{14} k^{-\frac12} h, \quad \mbox{if $l=3$}, \\
\label{eq4.29}
E\Bigl[\Bigl\|\int_0^{t_m} p(s)\,ds -k\sum^m_{n=1}p^n_h \Bigr\|_{L^2}\Bigr]
&\leq C_9 k^{\frac12}+C_{15} k^{-\frac12} h \quad \mbox{for $1\leq l\leq 3$}. 
\end{align}
\end{theorem}

\begin{remark}\label{rem7}
(a) It is clear from the above derivation that the only property of the Taylor-Hood mixed finite
element which is used in our analysis is its stability property, hence, the Taylor-Hood
element can be replaced by any stable mixed finite element (such as the MINI element), the analysis
still holds without any change.

(b) Since the above error bounds are of the order $O(k^{-\frac12} ( k+h) )$,  they suggest that the
	balanced choices of the mesh parameters are  $k\approx h$.

(c) The error bounds $E\bigl(\underset{{1\leq n\leq N}}\max\|u(t_n)-u^n_h\|^2_{L^2}\bigr)$,
$E\bigl(\underset{{1\leq n\leq N}}\max\|\nabla(u(t_n)-u^n_h)\|^2_{L^2}\bigr)$,
and $E\bigl(\underset{{1\leq n\leq N}}\max\|A(u(t_n)-u^n_h)\|^2_{L^2}\bigr)$
can also be derived using our analysis techniques.
\end{remark}

\section{Numerical experiments}\label{sec-5}
In this section, we present two $2$-D numerical tests to validate our theoretical
error estimates and to gauge the performance of the proposed fully discrete mixed method.
The experiments have been performed using the software package FreeFem++ \cite{H08}
in conjunction with UMFPACK \cite{Davis2004} and MATLAB on an
Intel(R) Core(TM) i7-8700 CPU @3.20 GHz with 16GB RAM.

\medskip
{\bf Test 1.}  In the first test, we take $D=(-1,1)^2$ and a deterministic constant force term $f$,
as well as the initial condition $u_0=0$.
In addition, $W$ in \eqref{eq1.1} is taken as a finite-dimensional $Q$-Wiener
process and $B(t,u(t))\equiv (u(t)^2+1)^{\frac{1}{2}} \in L^2(L^2(D), [L^2(D)]^d)$ such that
\[
\Bigl(\int_0^t B(s,u(s)) \,dW(s), \phi\Bigr) =\sum_{j,k=1}^M \lambda^{\frac{1}{2}}_{j,k}  \int_0^t \bigl(\bigl(u(s)^2+1\bigr)^{\frac{1}{2}}q_{j,k}, \phi \bigr)\, d\beta_{j,k}(s)
\qquad\forall \phi\in [L^2(D)]^d,
\]
where $\lambda_{j,k}=\frac{1}{(j+k)^2}\|g_{j,k}\|_{L^2}$ and
$$g_{j,k}(x,y):=\Bigl(c\bigl(\sin(j\pi x)+(j\pi x)^3\bigr)e^{-k\pi y},c\bigl(\cos(j\pi y)+(j\pi y)^3\bigr)e^{-k\pi x}\Bigr).$$
The orthogonal functions $q_{j,k}$ are defined by $q_{j,k}=:\frac{g_{j,k}}{\|g_{j,k}\|_{L^2}}$.
Hence, we take $Q=\mbox{span}\{q_{1,1},\cdots,q_{M,M}\}\subset H^1(D)$.
The lower order Taylor-Hood ($P_2-P_1$) element is employed.
Moreover, the following parameters are used for the test: $c=0.1$, $M=10$, $\nu=1$, $T=1$, $h=\frac{1}{40}$,
$k_0 = 1/10240$ (the minimum time step). Note that the change of the given functions $g_{j,k}(x,y)$
 is very dramatic in domain $D$, hence, it is reasonable to set $M=10$. The classical Monte Carlo method
with $N_p = 6000$ realizations is used to compute the expectation.
For any $1\leq n \leq N$, we use the following numerical integration formulas
\begin{align*}
AU^n &:=\Bigl(E\Bigl[\|U_{k_0}^n-U_{k_i}^n\|^2_{L^2}\Bigr]\Bigr)^{\frac12} \approx
\Bigl(\dfrac{1}{N_p}\sum_{\ell=1}^{N_p}\|U_{k_0}^n(\omega_\ell)-U_{k_i}^n(\omega_\ell)\|^2_{L^2}\Bigr)^{\frac12},\\
BU^n &:=\Bigl(E\Bigl[\|\nabla (U_{k_0}^n-U_{k_i}^n)\|^2_{L^2})\Bigr]\Bigr)^{\frac12} \approx
\Bigl(\dfrac{1}{N_p}\sum_{\ell=1}^{N_p}\|\nabla (U_{k_0}^n(\omega_\ell)-U_{k_i}^n(\omega_\ell))\|^2_{L^2}\Bigr)^{\frac12}, \\
AP^n:&=\Bigl(E\Bigl[\Big\|\int_0^Tp(s)ds-k_i\sum_{n=1}^{\frac{T}{k_i}}p_{h}^n\|^2_{L^2}\Bigr]\Bigr)^{\frac12}\\
&\approx
\Bigl(\dfrac{1}{N_p}\sum_{\ell=1}^{N_p}\Big\|k_0\sum_{n=1}^{\frac{T}{k_0}} p_{h}^n(\omega_\ell) -k_i\sum_{n=1}^{\frac{T}{k_i}}p_h^n(\omega_\ell)\Bigr\|^2_{L^2}\Bigr)^{\frac12},
\end{align*}
and
\begin{align*}
BP^n &:=\Bigl(E\Bigl[\|p_{k_0}^n-p_{k_i}^n\|^2_{L^2}\Bigr]\Bigr)^{\frac12} \approx
\Bigl(\dfrac{1}{N_p}\sum_{\ell=1}^{N_p}\|p_{k_0}^n(\omega_\ell)-p_{k_i}^n(\omega_\ell)\|^2_{L^2}\Bigr)^{\frac12}
\end{align*}
to approximate strong norms. An iterative linear solver based on the artificial
compressibility technique (cf. \cite{Ern_Guermond04}) is used to solve the linear system at each time step.

Figure \ref{fig1} displays the $L^2$ and $H^1$-norm errors ($AU^N$ and $BU^N$) of the time approximations
of the velocity
using different time step size $k$. It is clear that the numerical results verify the half order convergence
rate for the time discretization as predicted by the error analysis.

The left plot of Figure \ref{fig2} shows the $L^2$-norm error ($AP^N$) of the time-averaged pressure approximation
using different time step size $k$. The numerical results clearly verify the half order convergence rate as predicted
by the error analysis.
For curiosity and comparison purpose, we also present in the right plot of Figure \ref{fig2}
the standard $L^2$-norm error ($BP^N$) of the time approximations of the pressure using different time step size $k$.
The numerical results seem to suggest a convergence in that norm but with a much slow rate,
which is certainly caused by the low regularity of the pressure $p$.  It should be noted that
our convergence theory does not cover this case.

To verify the necessity of the error bound dependence on the factor $k^{-\frac1 2}$,
we fix $h=1/20$ and run the test again use different time step size $k$. The numerical results, presented in
Figure \ref{fig3}, show that both errors $AU^n$ and $AP^n$ increase as $k$ decreases, which proves that
both errors are inversely proportional to (a power of) the time step size $k$.
To verify the sharpness of the error bound dependence on the factor $k^{-\frac1 2}$, we
run the test again using $k\approx h$ for $(k,h)=(1/10, 1/8),(1/20,1/16), (1/40,1/32)$ and display the
numerical results in Figure \ref{fig4}. We observe that the numerical results show $O(k^{\frac12})$ order
convergence rate for the fully discrete scheme which exactly matches the theoretical rate predicted
by the error analysis.

\begin{figure}[th]
\centerline{
\includegraphics[scale=0.40]{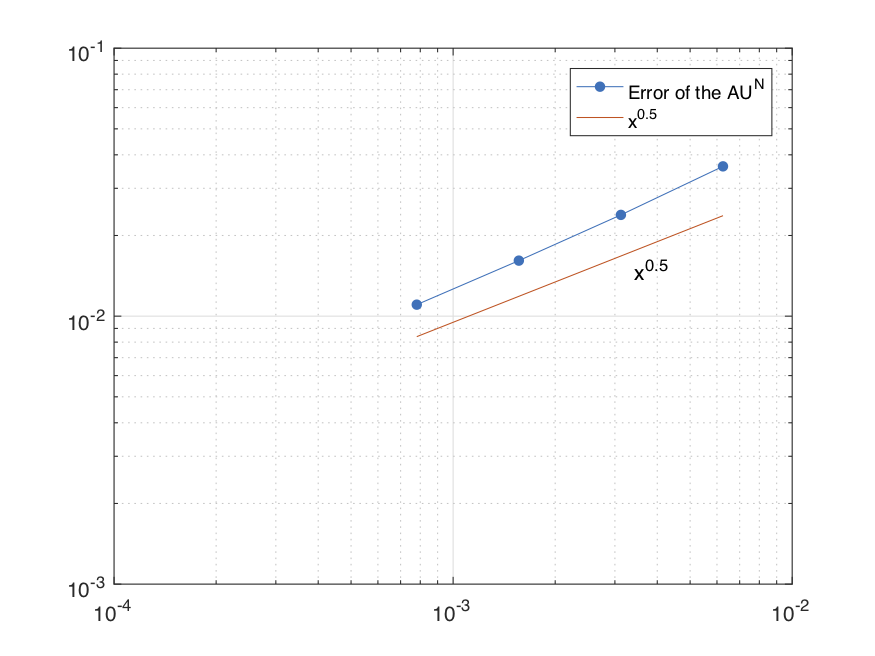}
\includegraphics[scale=0.40]{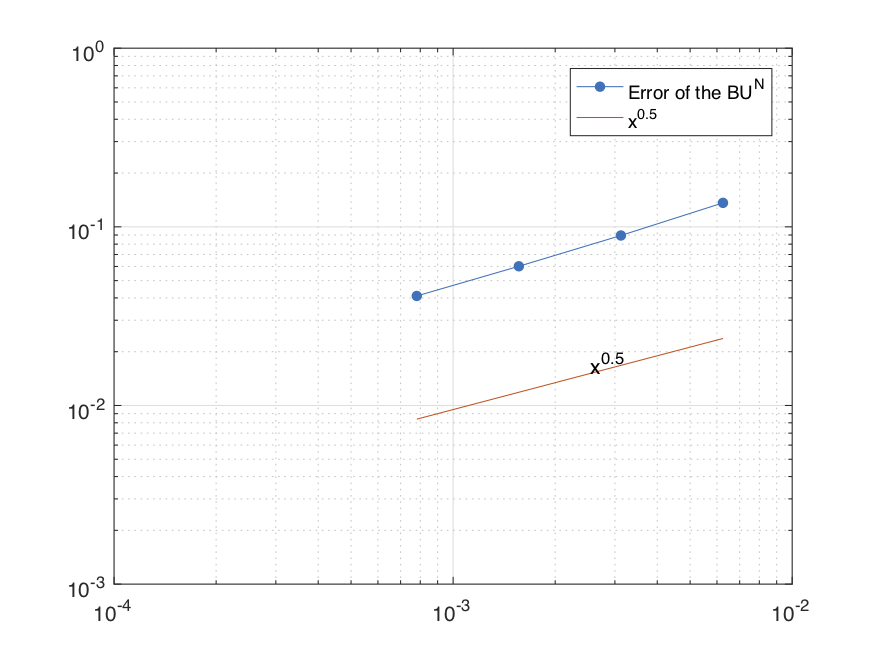}
}
\caption{Test 1: (a) The convergence rates $AU^N$; (b) the convergence rates $BU^N$.}
\label{fig1}
\end{figure}

\begin{figure}[th]
\centerline{
\includegraphics[scale=0.40]{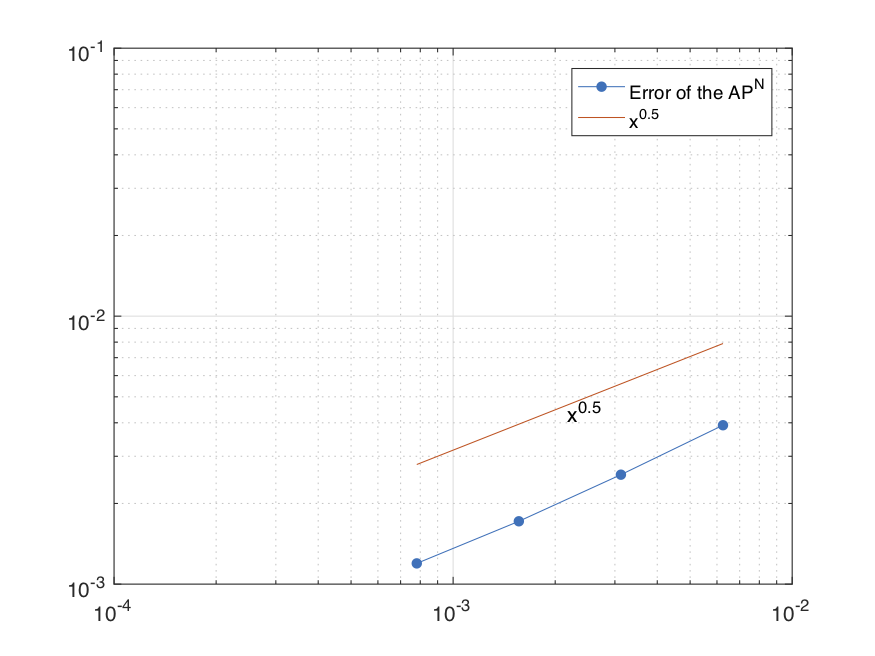}
\includegraphics[scale=0.40]{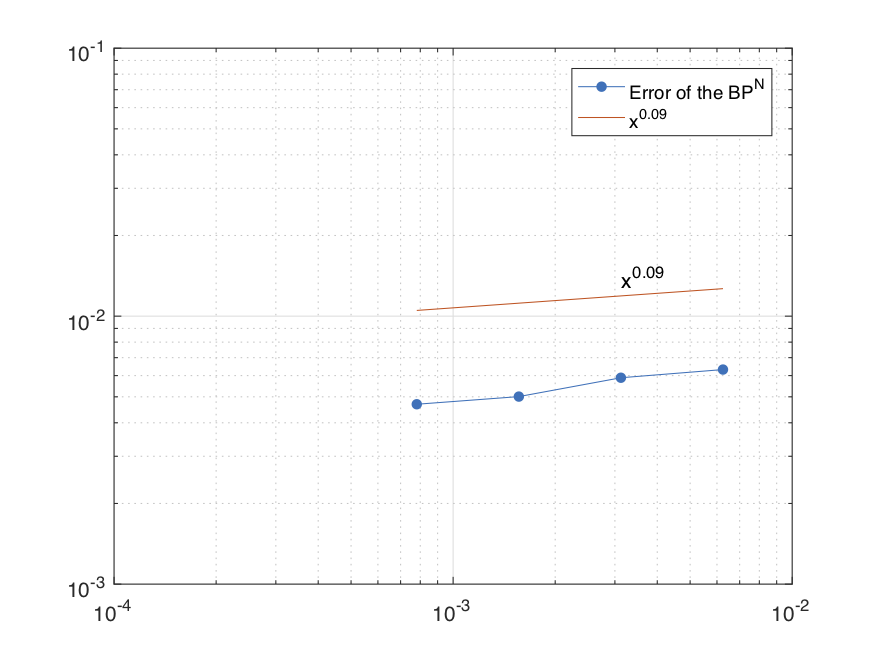}
}
\caption{Test 1: (a) The convergence rates $AP^N$; (b) the convergence rates $BP^N$.}
\label{fig2}
\end{figure}

\begin{figure}[th]
\centerline{
\includegraphics[scale=0.40]{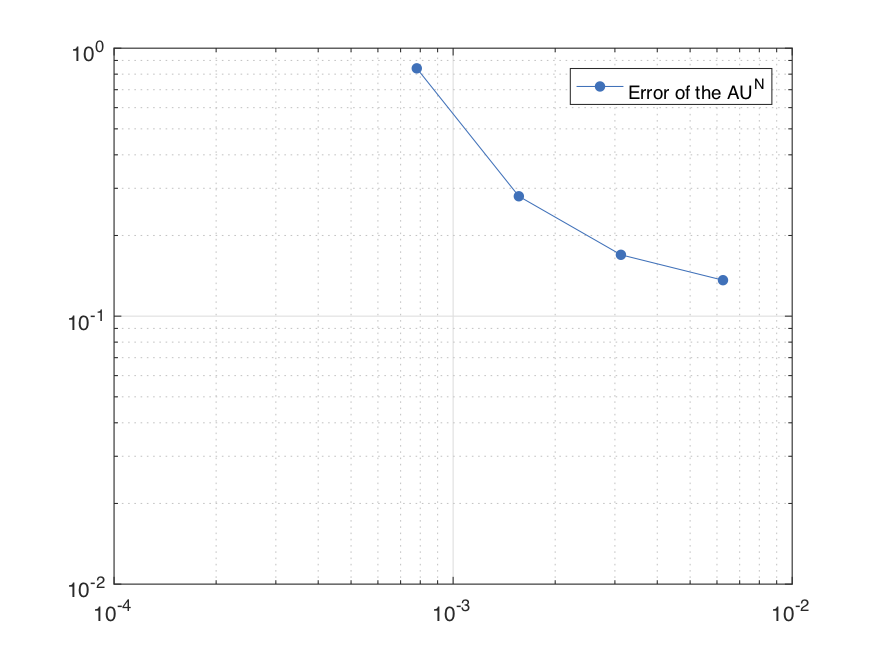}
\includegraphics[scale=0.40]{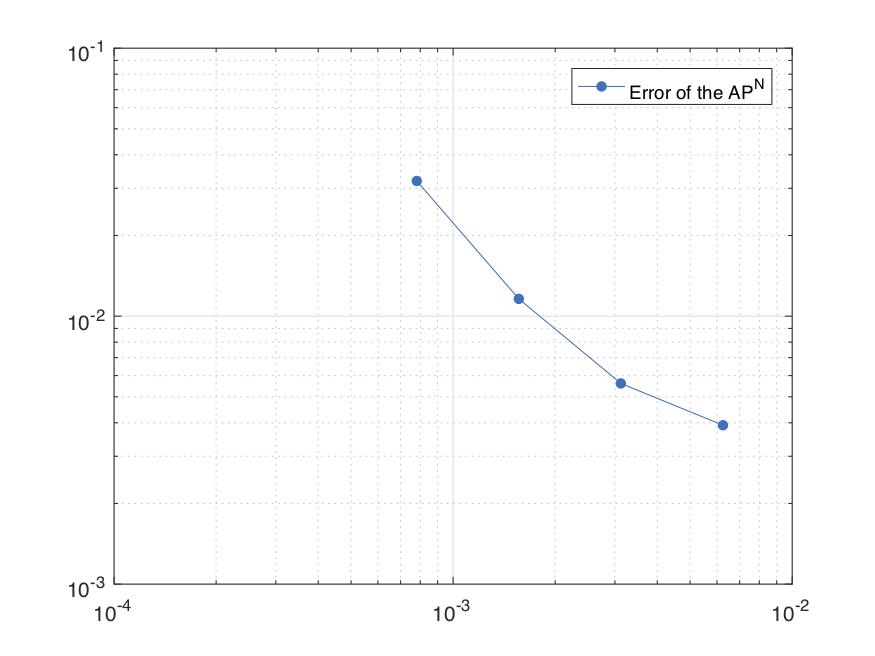}
}
\caption{Test 1: (a) The convergence rates $AU^N$; (b) the convergence rates $AP^N$.}
\label{fig3}
\end{figure}

\begin{figure}[th]
\centerline{
\includegraphics[scale=0.40]{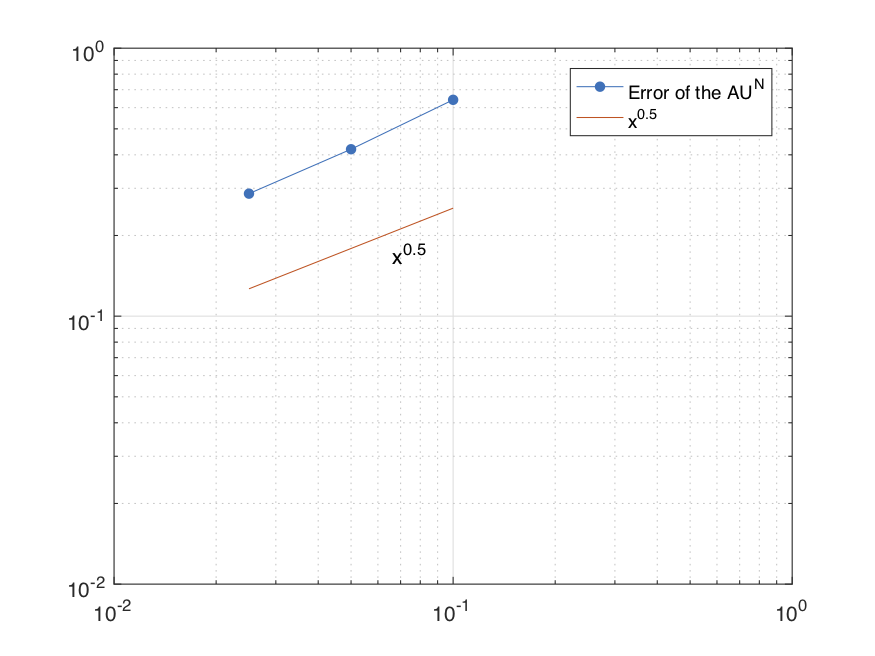}
\includegraphics[scale=0.40]{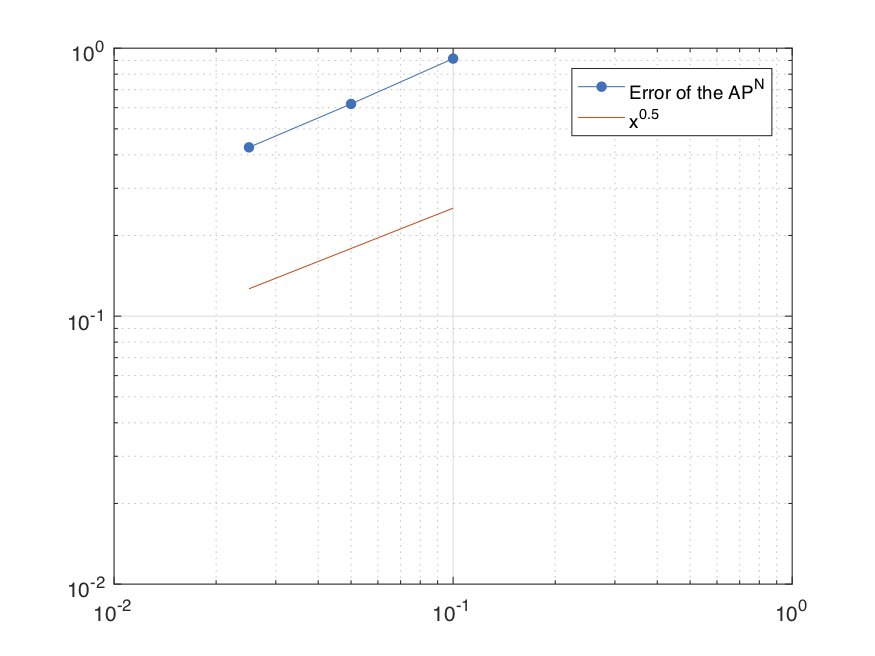}
}
\caption{Test 1: (a) The convergence rates $AU^N$; (b) the convergence rates $AP^N$.}
\label{fig4}
\end{figure}

\medskip
{\bf Test 2.} In the second numerical test, we compute the driven cavity flow on a unit square $(0,1)^2$.
In this test the force function $f$ is chosen to be the constant zero-vector $(0,0)$.
The no-slip boundary condition is only imposed on the part of the boundary $\{(x,1): 0<x<1\}$ with the
velocity $u=(1,0)$, and the zero Dirichlet condition is imposed on the rest of the boundary. The lowest order Taylor-Hood ($P_2-P_1$) element is used in this test.
The same finite-dimensional Q-Wiener process as in {\bf Test 1} is used and
we take $B(\cdot,u)= 1$ and use the following parameters: $c=1$, $M=10$, $\nu=1$, $T=1$, $h=\frac{1}{20}$,
$k = 0.005$, and the realization number $N_p = 5000$.

Figure \ref{fig5} plots (a) the expected value of the pressure $p_h^N$; (b) the expected value of the velocity
field $u_h^N$; (c) the streamlines of the expected value of $u_h^N$.
Figures \ref{fig6}--Fig \ref{fig8} show three computed samples of the pressure $p_h^N$ and the velocity
$u_h^N$  as  well as the streamlines of $u_h^N$.  From
Figure \ref{fig5}, we observe that the expectations of the pressure and velocity fields behave
similarly to their deterministic counterparts, on the other hand, Figures \ref{fig6}--Fig \ref{fig8}
show that the stochastic pressure and velocity samples could be very different from their deterministic fields.

\begin{figure}[h]
\centerline{
\includegraphics[scale=0.25]{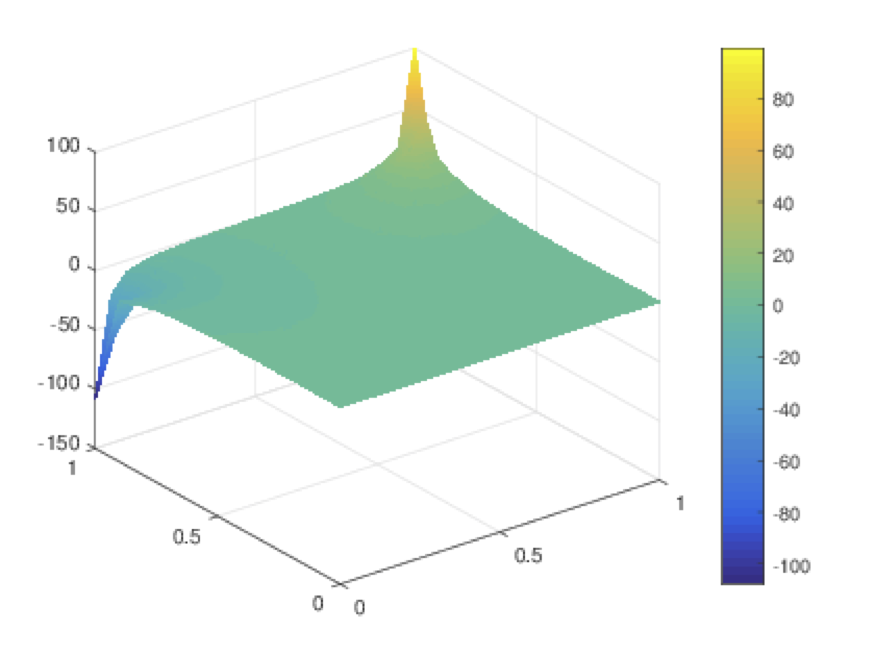}
\includegraphics[scale=0.30]{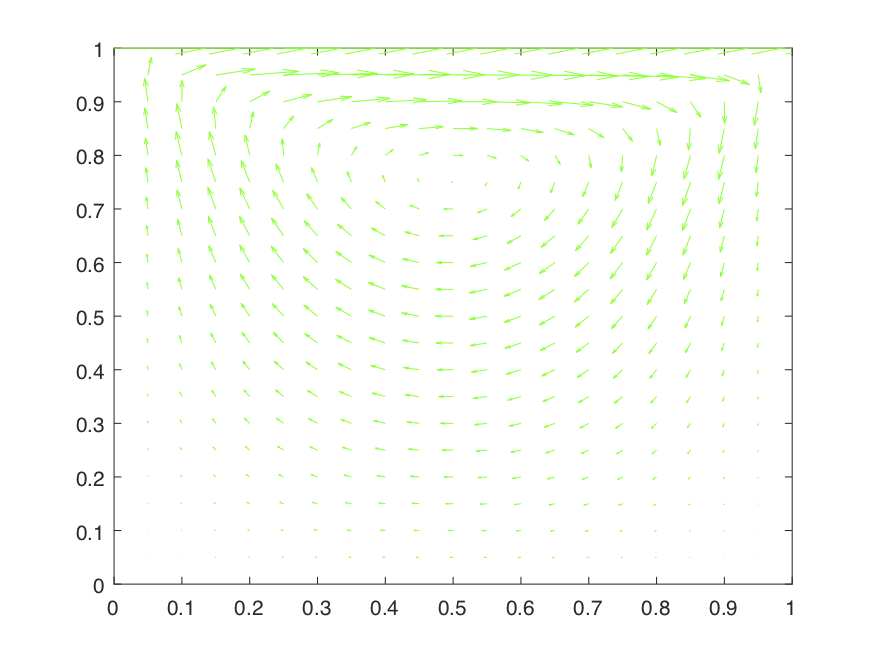}
\includegraphics[scale=0.30]{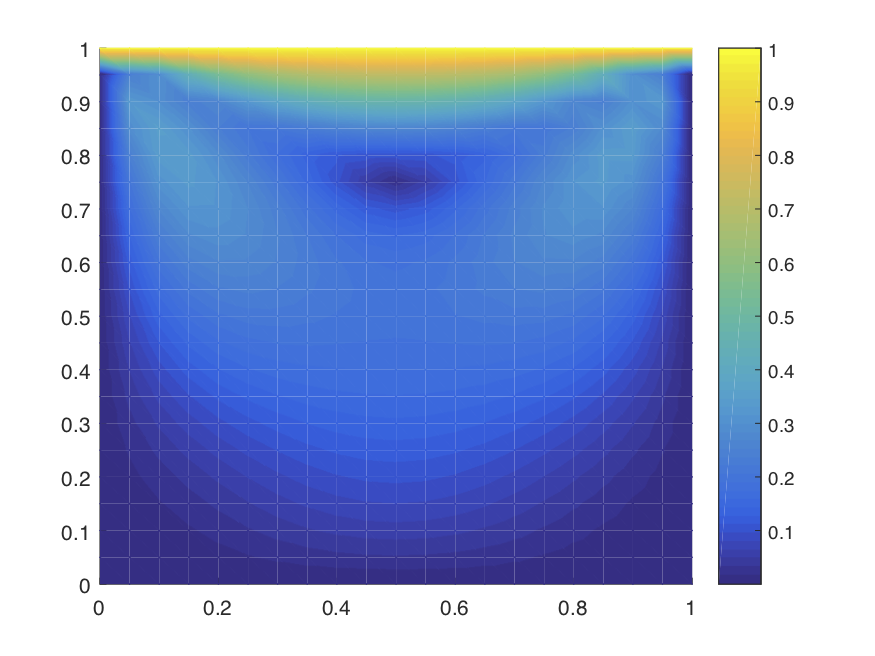}
}
\caption{Test 2: (a) The expected value of $p_h^N$; (b) the expected value of $u_h^N$;  (c)  the streamlines of the expected value of
	$u_h^N$.}
\label{fig5}
\medskip
\centerline{
\includegraphics[scale=0.25]{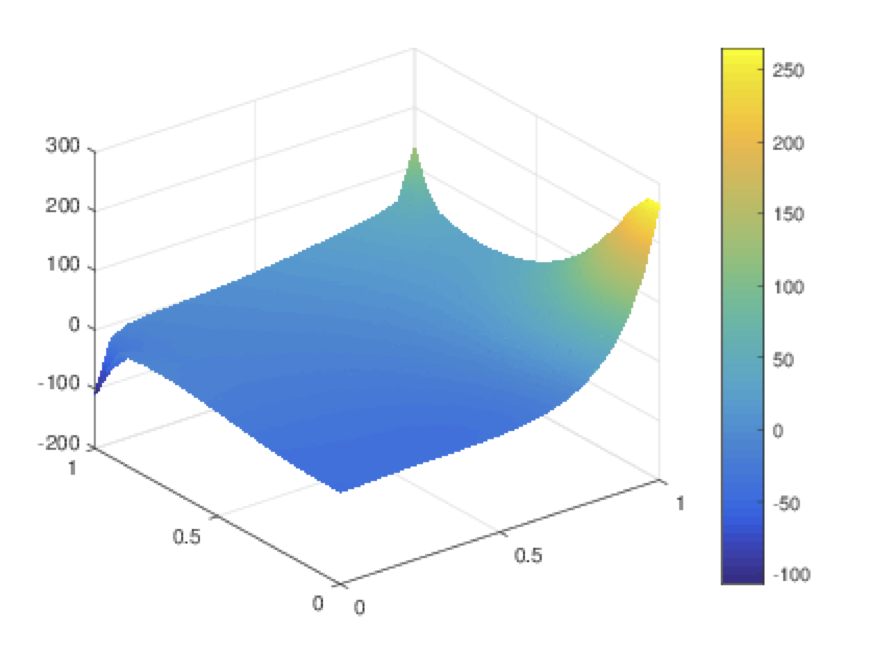}
\includegraphics[scale=0.30]{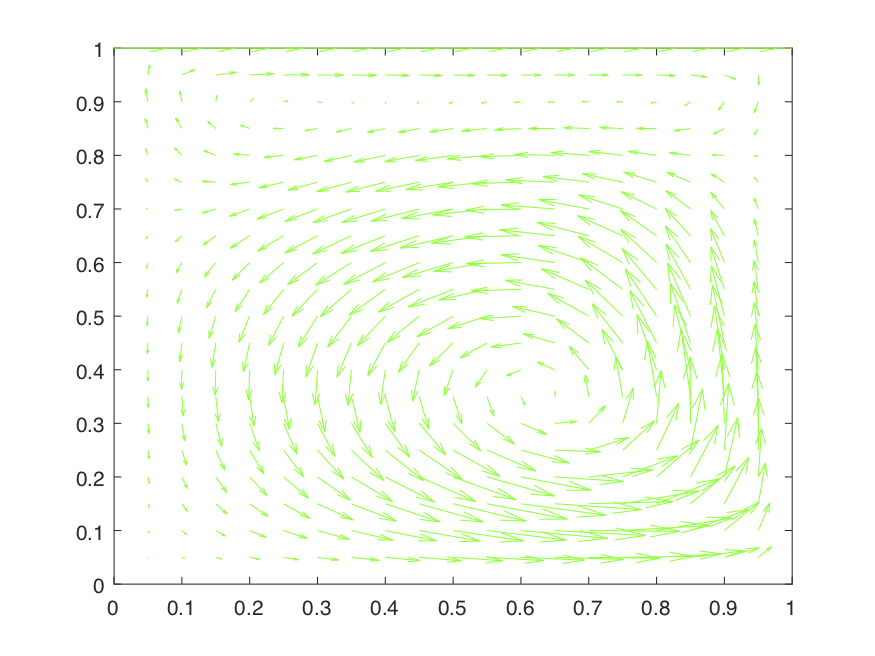}
\includegraphics[scale=0.30]{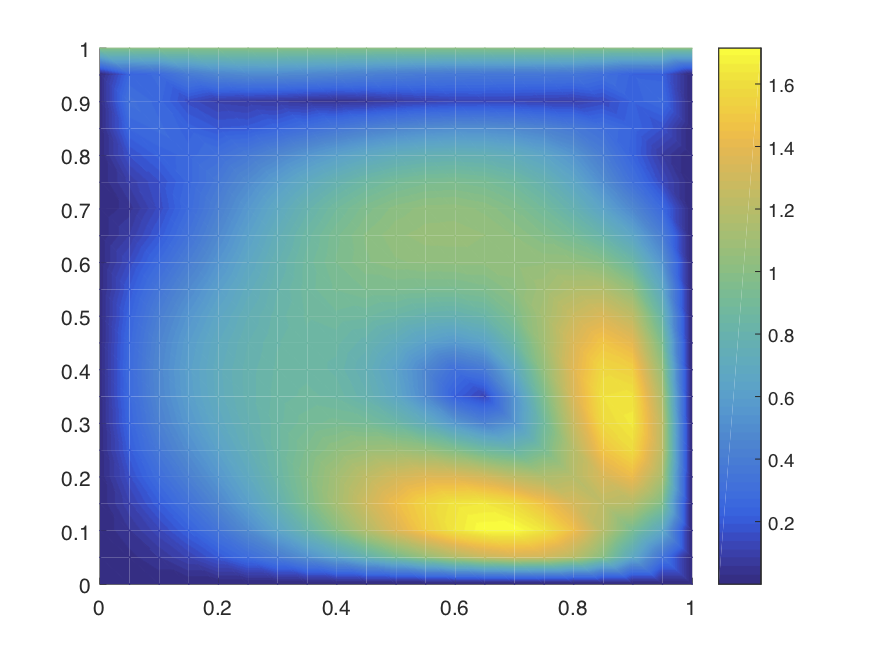}
}
\caption{Test 2:  The  first realization of  (a)  $p_h^N$; (b) $u_h^N$;  (c)  the streamlines of $u_h^N$.}
\label{fig6}
\medskip
\centerline{
\includegraphics[scale=0.25]{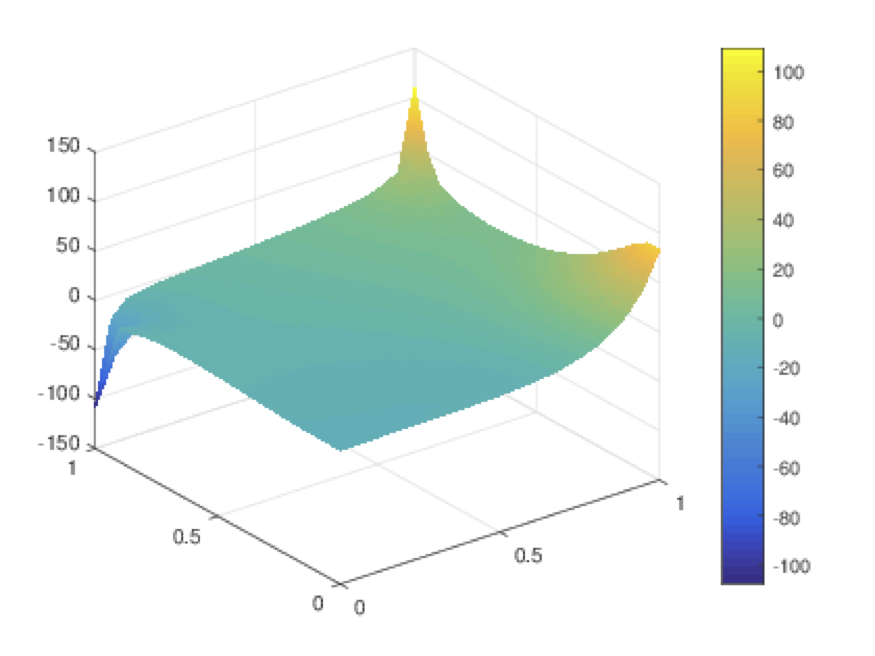}
\includegraphics[scale=0.30]{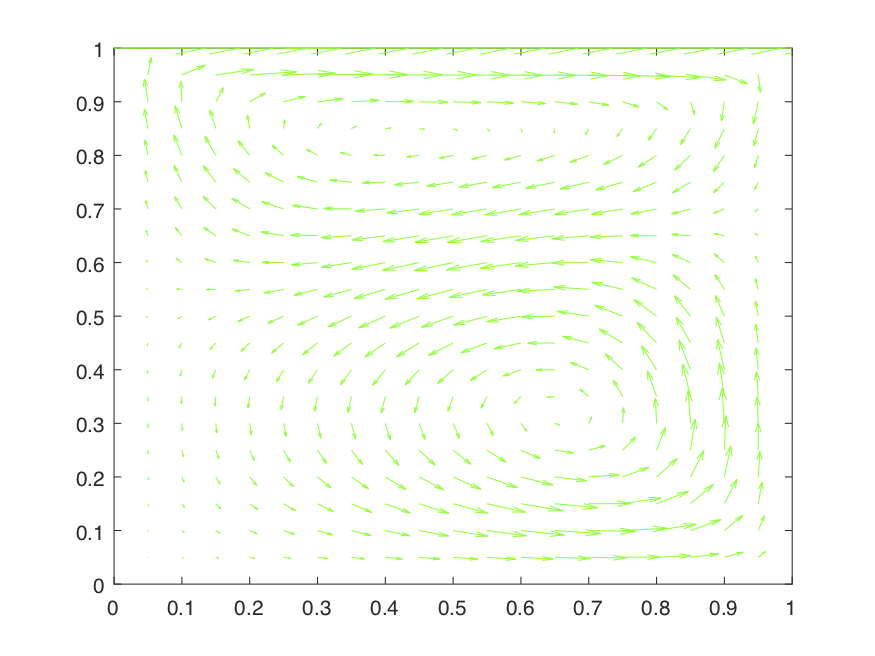}
\includegraphics[scale=0.30]{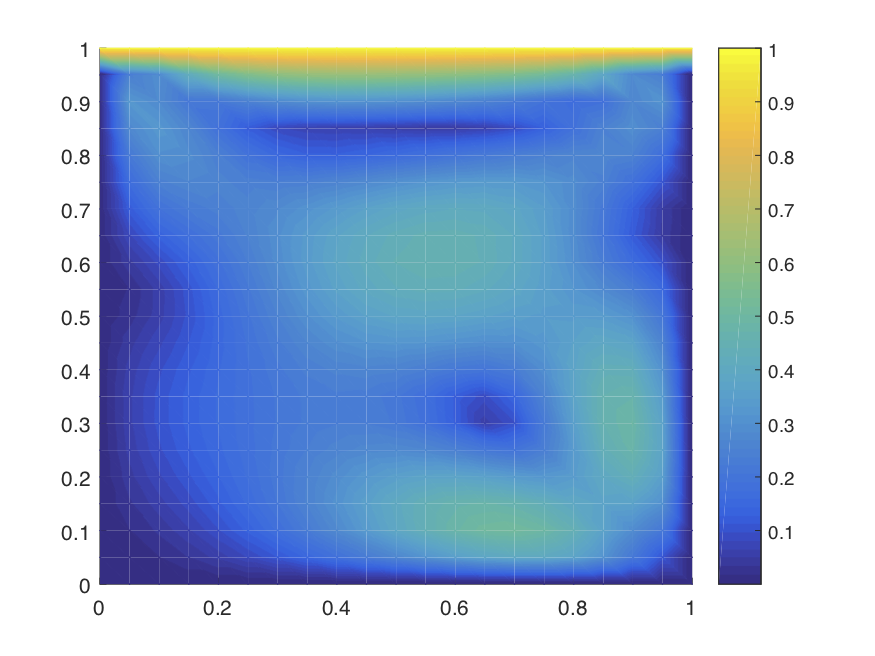}
}
\caption{Test 2:  The  second realization of  (a)  $p_h^N$; (b) $u_h^N$;  (c)  the streamlines of $u_h^N$.}
\label{fig7}
\medskip
\centerline{
\includegraphics[scale=0.25]{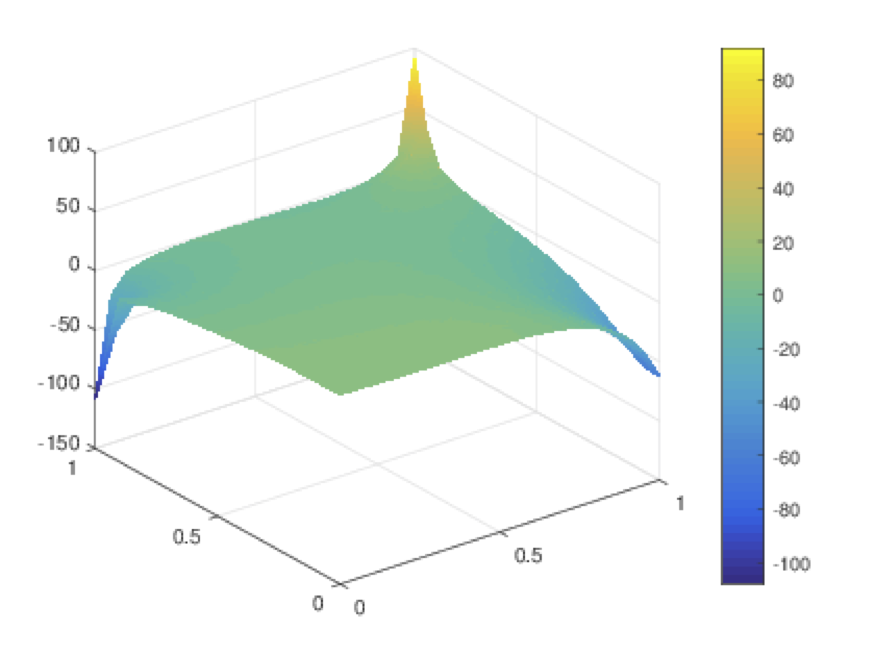}
\includegraphics[scale=0.30]{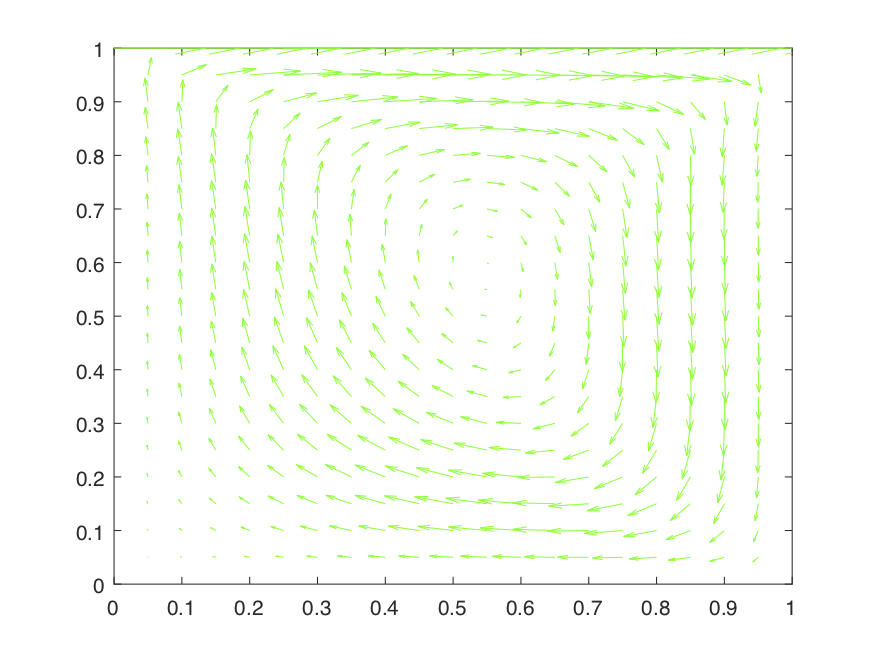}
\includegraphics[scale=0.30]{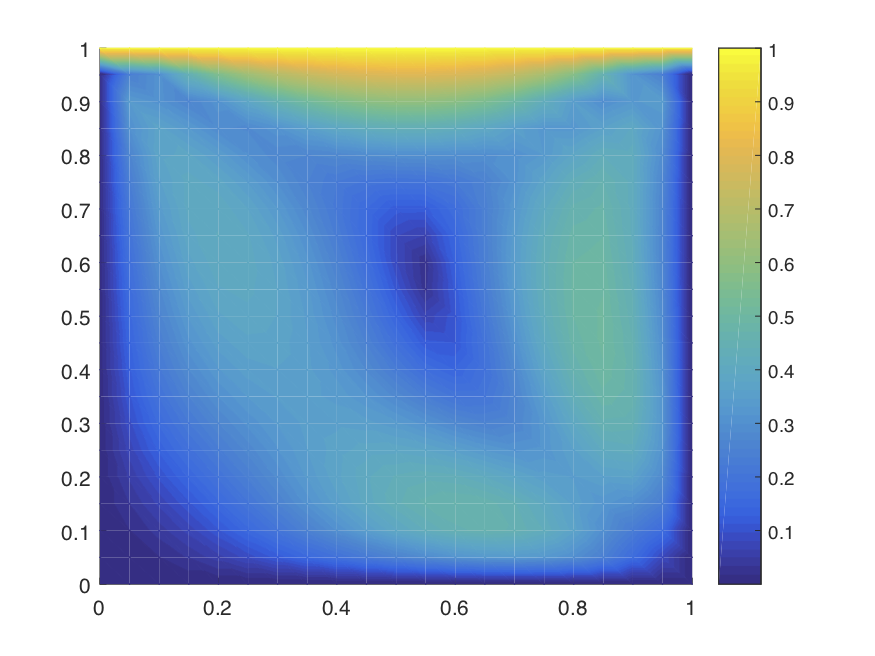}
}
\caption{Test 2: The  third realization of  (a)  $p_h^N$; (b) $u_h^N$;  (c)  the streamlines of $u_h^N$. }
\label{fig8}
\end{figure}

\bigskip
\textbf{Acknowledgments.} The authors would like to thank Professor Andreas Prohl of University of T\"ubingen (Germany)
for his many stimulating discussions and critical comments as well as valuable suggestions which help to improve
the early version of the paper considerably. In addition, his help on introducing and explaining several relevant
references is also greatly appreciated.


\end{document}